\documentclass[11pt, reqno]{amsart}
\usepackage{amssymb, amsthm, amsmath, amsfonts}
\usepackage{graphics}
\usepackage{hyperref}
\usepackage{bbm}
\usepackage[all]{xy}
\usepackage{enumerate}
\usepackage[mathscr]{eucal}
\usepackage{enumerate}
\usepackage{thmtools}
\usepackage{xcolor}
\usepackage{ytableau}

\theoremstyle{plain}
\newtheorem{theorem}{Theorem}[section]
\newtheorem{lemma}[theorem]{Lemma}
\newtheorem{proposition}[theorem]{Proposition}

\newtheorem{corollary}[theorem]{Corollary}
\numberwithin{equation}{section}

\theoremstyle{definition}

\newtheorem{definition}[theorem]{Definition}
\newtheorem{example}[theorem]{Example}
\newtheorem{remark}[theorem]{Remark}

\DeclareMathOperator{\Mod}{-Mod}
\DeclareMathOperator{\lsets}{-Set}
\DeclareMathOperator{\module}{-mod}
\DeclareMathOperator{\fdmod}{-fdmod}
\DeclareMathOperator{\tmod}{-tmod}
\DeclareMathOperator{\dMod}{-dMod}

\DeclareMathOperator{\Ob}{Ob}

\DeclareMathOperator{\cod}{cod}
\DeclareMathOperator{\coind}{coind}
\DeclareMathOperator{\ind}{ind}
\DeclareMathOperator{\res}{res}
\DeclareMathOperator{\Ext}{Ext}
\DeclareMathOperator{\Hom}{Hom}

\newcommand{\BC}{\boldsymbol{\C}}
\newcommand{\C}{{\mathscr{C}}}
\newcommand{\Co}{{\mathscr{C}^{\mathrm{op}}}}
\newcommand{\D}{{\mathscr{D}}}
\newcommand{\N}{{\mathbb{N}}}
\newcommand{\Sets}{{\mathrm{Set}}}
\newcommand{\op}{{\mathrm{op}}}

\newcommand{\sat}{{J\text{-}\mathrm{sat}}}
\newcommand{\Sh}{{\mathrm{Sh}}}
\newcommand{\sh}{{\mathrm{sh}}}
\newcommand{\PSh}{{\mathrm{PSh}}}
\newcommand{\FI}{{\mathrm{FI}}}
\newcommand{\VI}{{\mathrm{VI}}}
\newcommand{\OI}{{\mathrm{OI}}}
\newcommand{\Orb}{{\mathrm{Orb}}}

\title[A torsion theoretic interpretation for sheaves of modules]{A torsion theoretic interpretation for sheaves of modules and Grothendieck topologies on directed categories}

\author{Zhenxing Di}
\address{Fujian Province University Key Laboratory of Computation Science, School of Mathematical Sciences, Huaqiao University, Quanzhou 362021 China}
\email{dizhenxing@163.com}

\author{Liping Li}
\address{LCSM(Ministry of Education), School of Mathematics and Statistics, Hunan Normal University, Changsha 410081, China}
\email{lipingli@hunnu.edu.cn}

\author{Li Liang}
\address{Department of Mathematics, Gansu Provincial Basic Discipline Research Center for Complex Systems Analysis and Control, Lanzhou Jiaotong University, Lanzhou 730070, China}
\email{lliangnju@gmail.com}

\thanks{Z.X. Di was partly supported by NSF of China (Grant No. 12471034), Scientific Research Fund of Fujian Province (Grant No. 605-52525002) and
Scientific Research Fund of Huaqiao University (Grant No. 605-50Y22050); L.P. Li was partly supported by NSF of China (Grant Nos. 12571037 and 12471034); L. Liang was partly supported by NSF of China (Grant No. 12271230) and NSF of Gansu Province (Grant No. 26JRRA059). The authors are grateful to the anonymous referee for a careful and thorough reading of the manuscript and for numerous valuable suggestions that have significantly improved both the exposition and the quality of the paper.}

\keywords{Grothendieck topologies, rigid topologies, sheaves, torsion pairs, sheafification, sheaf cohomology, EI categories}

\begin{document}

\begin{abstract}
We prove that every Grothendieck topology induces a hereditary torsion pair in the category of presheaves of modules on a ringed site, and obtain a homological characterization of sheaves of modules: a presheaf of modules is a sheaf of modules if and only if it is saturated with respect to torsion presheaves, or equivalently, it is right perpendicular to torsion presheaves in the sense of Geigle and Lenzing. We also study Grothendieck topologies on directed categories $\C$, and show that every Grothendieck topology on $\C$ is a subcategory topology if and only if $\C$ is an artinian EI category. Consequently, in this case every sheaf category is equivalent to the presheaf category over a full subcategory of $\C$. Finally, we classify all Grothendieck topologies on a special type of noetherian EI categories, extend a few fundamental representation theoretic properties of $\FI$ and $\VI$ to their infinite full subcategories, and obtain a sheaf-theoretic reformulation of Alperin's weight conjecture.
\end{abstract}

\maketitle

\section{Introduction}

Sheaves of modules on ringed sites have been widely applied in algebraic geometry, algebraic topology, and geometric representation theory. Recently, they were also used to investigate modular representation theory of finite groups (see for instance \cite{Bal, WX, XX, XZ}) and continuous representations of topological groups \cite{DLLX}. This new approach combines ideas and methods from topos theory, representation theory of categories, and representation theory of groups, and hence not only proposes a comprehensive technique to study representations of groups, but also reveals deep relations among distinct areas. A classical example is Artin's theorem \cite[III. 9, Theorems 1 and 2]{MM}, from which it follows that the category of continuous representations of a topological group $G$ is equivalent to the category of sheaves of modules over a ringed site whose underlying category is an orbit category of $G$ equipped with the atomic Grothendieck topology.

Sheaf theory is expected to play a more significant role in algebraic representation theory. To facilitate its applications in this area, it is useful to reinterpret concepts and notions in topos theory, which are formal and sometimes hard to check in practice, in terms of representation theory. In \cite{DLLX} we give an attempt for the special case of the atomic Grothendieck topology via a particular torsion theory. The first goal of this paper is to extend it to the full generality of all Grothendieck topologies.

Let us give some details. Let $\C$ be a small category, $J$ a Grothendieck topology on the opposite category $\Co$,\footnote{In this paper we consider representations of $\C$; that is, covariant functors from $\C$ to a concrete category. Since they are precisely presheaves over $\Co$, we impose Grothendieck topologies on $\Co$ rather than on $\C$.} and $\mathcal{O}$ a structure sheaf over the site $\BC^{\op} = (\Co, J)$. Given a presheaf $V$ of modules over the ringed site $(\BC^{\op}, \mathcal{O})$  (also called an \textit{$\mathcal{O}$-module}) as well as an object $x$ in $\C$, we define an element $v \in V_x$, the value of $V$ on $x$, to be \textit{$J$-torsion} if there is a covering sieve $S \in J(x)$ such that every morphism in $S$ sends $v$ to 0. The axioms of Grothendieck topologies guarantee that all $J$-torsion elements in $V$ form an $\mathcal{O}$-submodule; see Lemma \ref{torsion submodules}. Consequently, the operation of taking the maximal $J$-torsion submodule gives rise to a left exact endo-functor $\mathcal{T}_J$ on $\mathcal{O} \Mod$, the category of $\mathcal{O}$-modules. We call this functor the \textit{torsion functor} and say that $V$ is \textit{$J$-torsion} (resp., \textit{$J$-torsion free}) if $\mathcal{T}_J(V) = V$ (resp., $\mathcal{T}_J(V) = 0$). If $V$ is $J$-torsion free and $\mathrm{R}^1 \mathcal{T}_J(V) = 0$, where $\mathrm{R}^1 \mathcal{T}_J$ is the first right derived functor of $\mathcal{T}_J$, then we call it a \textit{$J$-saturated} module.

Denote by $\mathcal{T}(J)$ (resp., $\mathcal{F}(J)$) the category of $J$-torsion (resp., $J$-torsion free) $\mathcal{O}$-modules. We show that $(\mathcal{T}(J), \mathcal{F}(J))$ is a hereditary torsion pair on $\mathcal{O} \Mod$, and it completely determines the Grothendieck topology. That is, if $K$ is another Grothendieck topology on $\Co$ with respect to which $\mathcal{O}$ is still a structure sheaf, then $J = K$ if and only if $\mathcal{T}(J) = \mathcal{T}(K)$. For details, see Propositions \ref{bijective correspondence} and \ref{torsion theory}.

The following result gives a characterization of sheaves of modules over $(\BC^{\op}, \mathcal{O})$ in terms of the torsion theory mentioned above, and extends \cite[Theorem 3.8]{DLLX} to full generality. Compared to the original definition of sheaves in terms of amalgamations of matching families, this characterization is homological, and is more elementary for applications in representation theory.

\begin{theorem}[Theorem \ref{characterize sheaves}] \label{Theorem 1}
An $\mathcal{O}$-module $V$ is a sheaf of modules on $(\BC^{\op}, \mathcal{O})$ if and only if it is $J$-saturated.
\end{theorem}

It is direct to check that $J$-saturated $\mathcal{O}$-modules are precisely $\mathcal{O}$-modules right perpendicular to $J$-torsion $\mathcal{O}$-modules in the sense of Geigle-Lenzing \cite{GLen}; for details, see \cite[Section 1]{GLen}. Consequently, applying \cite[Propositions 2.2 and 2.5]{GLen}, one immediately deduces the following result:

\begin{corollary}[Corollary \ref{equivalence}] \label{Theorem 2}
One has the following equivalence
\[
\Sh(\BC^{\op}, \mathcal{O}) \simeq \mathcal{O} \Mod / \mathcal{T}(J),
\]
namely the category of sheaves of modules on $(\BC^{\op}, \mathcal{O})$ is equivalent to the Serre quotient of the category of $\mathcal{O}$-modules by the category of $J$-torsion $\mathcal{O}$-modules.
\end{corollary}

\begin{remark}
It is well known that the sheaf category is a localization of the presheaf category by a localizing subcategory. This corollary gives an explicit description of the localizing subcategory: the full subcategory consists of all $J$-torsion modules. Relations between Grothendieck topologies and torsion theories have also been studied in the settings of linear sites and modules over small preadditive categories; see \cite{Lowen, PSV}. In particular, Lowen's approach \cite{Lowen} is formulated in terms of linear topologies on $k$-linear categories. In contrast, our results apply directly to arbitrary ringed sites via constructing the associated torsion theory intrinsically from the Grothendieck topology, and obtain a homological characterization of sheaves of modules.
\end{remark}

The above homological characterization of sheaves of modules and equivalence of categories bring quite a few useful consequences. For instance, injective objects in the sheaf category are precisely $J$-torsion free injective $\mathcal{O}$-modules (see Corollary \ref{injective sheaves}); sheaf cohomology can be reinterpreted via the torsion functor $\mathcal{T}_J$ (see Corollary \ref{sheaf cohomology}). Furthermore, one can obtain a very elementary description of the sheafification functor (see the explanation after Corollary \ref{equivalence}).

The Serre quotient still seems mysterious for practical purpose. For instance, even for the special case that $\Co$ is an orbit category of a finite group, $J$ is the atomic Grothendieck topology, and $\mathcal{O}$ is a constant structure sheaf induced by a field, in general it is a hard task to classify irreducible objects in $\mathcal{O} \Mod / \mathcal{T}(J)$ (as we will see later, an answer to this question can bring a significant progress for verifying Alperin's weight conjecture). The second goal of this paper is to find some specific categories such that one can obtain a relatively explicit description of this Serre quotient.

It has been pointed out in \cite[IV.9]{AGV} by Grothendieck and Verdier that every Grothendieck topology on $\Co$ is \textit{rigid} (or called a \textit{subcategory topology}) when $\C$ is a finite Karoubi category, and in this case $\Sh(\BC^{\op}, \mathcal{O})$ is equivalent to the category of presheaves of modules on a full subcategory of $\Co$. Similar results have been established by Hemelaer \cite{Hem} and Lindenhovius \cite{Lin} for posets satisfying certain extra conditions. Although it seems to be a very difficult task to classify small categories whose Grothendieck topologies are rigid, in this paper we obtain a satisfactory answer for categories satisfying the following combinatorial condition: the binary relation $\leqslant$ on $\Ob(\C)$ defined by setting $x \leqslant y$ if and only if $\C(x, y) \neq \varnothing$ is a partial order. We call them \textit{directed categories}. The reason we choose them is plain: many important categories of particular interest in representation theory are directed categories, for example, posets, orbit categories of groups, quivers without loops or oriented cycles, etc.

\begin{theorem}[Proposition \ref{necessary direction} and Corollary \ref{sufficient direction}] \label{Theorem 3}
Let $\C$ be a directed category. Then every Grothendieck topology on $\Co$ is rigid if and only if $\C$ is a noetherian EI category; that is, the poset $(\Ob(\C), \leqslant)$ satisfies the ascending chain condition and every endomorphism in $\C$ is an isomorphism.
\end{theorem}

\begin{remark}
Recently, Marqu\`es \cite{Marques} obtained a criterion for stable universal rigidity of arbitrary small categories, requiring all Grothendieck topologies on their slice categories to be rigid. In contrast, Theorem \ref{Theorem 3} gives a simple characterization of universal rigidity for opposites of directed categories in terms of the noetherian EI condition. In this setting, universal rigidity and stable universal rigidity coincide.
\end{remark}

Consequently, when $\C$ is a noetherian EI category, questions concerning $\Sh(\BC^{\op}, \mathcal{O})$ can be completely transferred to questions about representations of full subcategories of $\C$, so we have no ``generic" sheaf theory in this situation. Explicitly, by the Comparison Lemma \cite[C2.2, Theorem 2.2.3]{Jo} (see also \cite[C2.2, Example 2.2.4(d)]{Jo}), one has:

\begin{corollary}[Theorem \ref{classify sheaves}] \label{Theorem 4}
Let $\C$ be a noetherian EI category, $J$ a Grothendieck topology on $\Co$, $\mathcal{O}$ a structure sheaf over $\BC^{\op}$, and $\mathscr{D}$ the full subcategory of $\C$ consisting of objects $x$ with $J(x) = \{ \C(x, -) \}$. Then:
\begin{enumerate}
\item an $\mathcal{O}$-module $V$ lies in $\mathcal{T} (J)$ if and only if $V_x = 0$ for all $x \in \Ob(\D)$;

\item one has
\[
\Sh(\BC^{\op}, \, \mathcal{O}) \simeq \mathcal{O}_{\D} \Mod,
\]
where $\mathcal{O}_{\D}$ is the presheaf of rings on $\D^{\op}$ obtained via restricting $\mathcal{O}$ to $\D^{\op}$.
\end{enumerate}
\end{corollary}

However, when $\C$ is not noetherian, a classification of Grothendieck topologies on $\Co$ becomes much harder. In this paper we solve this question only for EI categories $\C$ of type $\mathbb{N}$ (see the beginning of Section 6 for a precise definition) or type $\mathbb{Z}$. Although it seems to the reader that too many combinatorial restrictions are imposed on them, there are still quite a few interesting examples whose representation theory is actively studied (see for instance \cite{CEF, CEFN, Nag1, Nag2, SS, SS1}):
\begin{itemize}
\item skeletal full subcategories of the category $\FI$ of finite sets and injections, where $\FI$ is equivalent to the opposite category of an orbit category of the permutation group on $\N$ (\cite[Examples A2.1.11(h) and D3.4.10]{Jo});

\item skeletal full subcategories of the opposite category of the category $\mathrm{FS}$ of finite sets and surjections, where $\mathrm{FS}$ is equivalent to an orbit category of the automorphism group of the free Boolean algebra on a countable infinity of generators, or the self-homeomorphism group of the Cantor set (\cite[Example D3.4.12]{Jo});

\item skeletal full subcategories of the category $\mathrm{VI}$ of finite dimensional vector spaces over a field $k$ and linear injections, where $\mathrm{VI}$ is equivalent to the opposite category of an orbit category of $\varinjlim_n \mathrm{GL}_n (k)$ by \cite[Subsection 5.3.1]{DLLX}.
\end{itemize}

We can classify all Grothendieck topologies on $\Co$ when $\C$ is an EI category of type $\mathbb{N}$ or $\mathbb{Z}$.

\begin{theorem}[Propositions \ref{generic topologies}, \ref{nongeneric topologies} and \ref{rigid topologies for type n}] \label{Theorem 5}
Let $\C$ be an EI category of type $\mathbb{N}$ or $\mathbb{Z}$. One has:
\begin{enumerate}
\item Grothendieck topologies $J$ on $\Co$ such that $J(x)$ does not contain the empty functor for each $x \in \Ob(\C)$ are parameterized by functions
\[
\boldsymbol{d}: \Ob(\C) \longrightarrow \N \cup \{\infty\}
\]
satisfying the following condition: if $\boldsymbol{d} (n) \neq 0$, then $\boldsymbol{d} (n + 1) = \boldsymbol{d} (n) - 1$.

\item All other Grothendieck topologies on $\Co$ are parameterized by functions
\[
\boldsymbol{d}: \Ob(\C) \longrightarrow \N \cup \{-\infty\}
\]
satisfying the following conditions: $\boldsymbol{d} (n) = -\infty$ for $n \gg 0$; and if $\boldsymbol{d} (n) \neq 0$, then $\boldsymbol{d} (n + 1) = \boldsymbol{d} (n) - 1$.

\item A Grothendieck topology $J$ on $\C^{\op}$ is rigid if and only if its corresponding function $\boldsymbol{d}$ satisfies the following condition: $\boldsymbol{d}(n) \neq \infty$ for all $n \in \N$.
\end{enumerate}
\end{theorem}

For an EI category $\C$ of type $\mathbb{N}$ or $\mathbb{Z}$, it follows that a Grothendieck topology $J$ on $\Co$ which is not a subcategory topology is almost atomic; that is, there is a full subcategory of $\Co$ containing all but finitely many objects such that the restriction of $J$ to this subcategory is the atomic Grothendieck topology. Accordingly, to describe sheaves of modules over arbitrary ringed sites $(\BC^{\op}, \, \mathcal{O})$, we only need to understand sheaves of modules over ringed atomic sites. For a precise meaning of this statement, see Proposition \ref{sheaves for type n} and Remark \ref{type Z}.

We can use sheaf theory to deduce some interesting results in representation theory of combinatorial categories. For instance, let $k$ be a field of characteristic 0, and let $\C$ be a skeleton subcategory of the category $\FI$ of finite sets and injections, or the category $\VI_q$ of finite dimensional vector spaces over a finite field $\mathbb{F}_q$ and $\mathbb{F}_q$-linear injections. It has been proved by Gan-Li \cite{GL} (see also Sam-Snowden \cite{SS} for $\FI$) that every indecomposable projective $\C$-module over $k$ is also injective. Surprisingly, this locally self-injective property and many other fundamental results of $\C$-modules established by Church-Ellenberg-Farb-Nagpal \cite{CEFN}, Gan-Li \cite{GL}, Gan-Li-Xi \cite{GLX}, Nagpal \cite{Nag1, Nag2}, and Sam-Snowden \cite{SS, SS1} extend to representations of its infinite full subcategories.

\begin{theorem}[Proposition \ref{noetherianity}, \ref{self-injectivity} and \ref{generalized representation stability}] \label{Theorem 6}
Let $\C$ be specified as above, $k$ a commutative noetherian ring, and $\D$ a full subcategory of $\C$ containing infinitely many objects. Then every finitely generated $\D$-module over $k$ is noetherian. Moreover, if $k$ is a field of characteristic 0, then one has
\begin{enumerate}
\item every indecomposable projective $\D$-module over $k$ is also injective;

\item the category $\D \module$ of finitely generated $\D$-modules has enough injectives;

\item up to isomorphism, an indecomposable injective $\D$-module is one of the following three:
\begin{itemize}
\item a summand of $k\D(x, -)$ for a certain $x \in \Ob(\D)$,
\item a summand of the restriction of $k\C(x, -)$ for a certain $x \in \Ob(\C) \setminus \Ob(\D)$,
\item a finite dimensional injective $\D$-module;
\end{itemize}

\item every finitely generated $\D$-module has finite injective dimension;

\item one has the following equivalence of categories
\[
\D \module / \D \fdmod \simeq \C \fdmod,
\]
where $\D \fdmod$ (resp., $\C \fdmod$) is the category of finite dimensional $\D$-modules (resp, $\C$-modules), and $\D \module / \D \fdmod$ is the Serre quotient.
\end{enumerate}
\end{theorem}

\begin{remark}
A standard tool to prove these results for $\C$ is the \textit{shift functor} introduced by Church-Ellenberg-Farb \cite{CEF} for representations of $\FI$ and Nagpal \cite{Nag1} for representations of $\VI_q$. Unfortunately, it does not work in this much wider situation as the objects in $\D$ in general do not occur in a periodic way, which is essential for people to define the shift functor. However, using sheaf theory, one can deduce this theorem via an almost transparent way from those established results of $\C$.
\end{remark}

This paper is organized as follows. In Section 2 for the reader's convenience we provide some background knowledge on Grothendieck topologies and sheaf theory. In Section 3 we construct a torsion theory induced by a rule $J$ satisfying partial properties of Grothendieck topology. Indeed, for most results described in this section we do not require $J$ to be a Grothendieck topology, so the reader can understand the role of each individual property. The main content of Section 4 consists of a proof of Theorem \ref{Theorem 1} and several consequences. Directed categories are introduced in Section 5, where we also give a proof of Theorem \ref{Theorem 3}. In Section 6 we classify Grothendieck topologies on $\Co$ when $\C$ is an EI category of type $\N$, proving Theorem \ref{Theorem 5}. In the last section we describe some potential applications of previously established results and prove Theorem \ref{Theorem 6}.

Throughout this paper, unless otherwise specified, all functors are covariant, modules over rings are left modules, and composition of maps and morphisms is from right to left.

\section{Preliminaries}

Let $\C$ be a small skeletal category and let $\Co$ be its opposite category. In this section we briefly recall definitions and elementary facts on Grothendieck topologies and sheaves. For more details, please refer to \cite{Ar, AGV, Jo, KW, MM, Stack}.

\subsection{Grothendieck topologies}

Let $x$ be an object in $\C$. A \textit{left ideal} $S$ of morphisms on $x$ is a subfunctor of the representable functor $\C(x, -)$. Equivalently, $S$ is a set of morphisms with domain $x$ such that the following condition holds: if $f: x \to y$ is a member in $S$, and $g: y \to z$ is an arbitrary morphism in $\C$, then $gf = g \circ f$ is contained in $S$ as well. A left ideal of morphisms on $x$ is also called a \textit{sieve} on $x$ in $\Co$.

Throughout this paper we fix $J$ to be a rule assigning to each object $x$ a collection $J(x)$ of left ideals of morphisms on $x$. The following definition is taken from \cite[III.2, Definition 1]{MM}, which defines a \emph{Grothendieck topology} on $\Co$ if $J$ satisfies all three axioms. Here we use the dual version, since we mainly consider representations of $\C$ which are covariant functors, rather than presheaves which are contravariant functors.

\begin{definition} \label{def of Grothendieck topologies}
We say that $J$ satisfies
\begin{enumerate}
\item the \textit{maximal axiom} if $\C(x, -) \in J(x)$ for each $x \in \Ob(\C)$;

\item the \textit{stability axiom} if $S \in J(x)$ implies that the pullback
\[
f^{\ast} (S) = \{ g: y \to \bullet \mid gf \in S \} \in J(y)
\]
for any morphism $f: x \to y$ in $\C$, where $\bullet$ means an arbitrary object in $\C$;

\item the \textit{transitivity axiom} if $S \in J(x)$ and $T$ is a left ideal of morphisms on $x$ such that $f^{\ast}(T) \in J(y)$ for any $(f: x \to y) \in S$, then $T \in J(x)$.
\end{enumerate}
\end{definition}

The following fact is well known; see \cite[Section III.2]{MM}. For the convenience of the reader, we give a complete proof.

\begin{lemma}
If $J$ is a Grothendieck topology on $\Co$, then:
\begin{enumerate}
\item $J$ is closed under upward inclusions: if $S_1$ is a subfunctor of $\C(x, -)$, $S_2$ is a subfunctor of $S_1$ and $S_2 \in J(x)$, then $S_1 \in J(x)$ as well;
\item $J$ is closed under finite intersections: if $S_1, S_2 \in J(x)$, then $S_1 \cap S_2 \in J(x)$.
\end{enumerate}
\end{lemma}

\begin{proof}
For every $f: x \to y$ in $S_2$, one has $f^{\ast} (S_1) = \C(y, -) \in J(y)$, so (1) follows from the transitivity axiom. Similarly, for every $f: x \to y$ in $S_2$, one can check that $f^{\ast} (S_1 \cap S_2) = f^{\ast} (S_1)$, so $f^{\ast} (S_1 \cap S_2)$ is contained in $J(y)$ by the stability axiom. Consequently, $S_1 \cap S_2 \in J(x)$ by the transitivity axiom.
\end{proof}

\begin{example} \label{examples of topologies}
We list a few important Grothendieck topologies on $\Co$ below:
\begin{itemize}
\item the \textit{trivial topology} $J_t$: $J_t(x) = \{\C(x, -) \}$ for every $x \in \Ob(\C)$;

\item the \textit{maximal topology} $J_m$: $J_m(x)$ consists of all subfunctors (including $\varnothing$) of $\C(x, -)$;

\item the \textit{dense topology} $J_d$: a subfunctor $S \subseteq \C(x, -)$ is contained in $J_d(x)$ if for every morphism $f: x \to y$, there exists a morphism $g: y \to z$ such that $gf \in S$;

\item the \textit{atomic Grothendieck topology} $J_a$: $J_a(x)$ consists of all nonempty subfunctors of $\C(x, -)$ provided that $\C$ satisfies the following Ore condition: given two morphisms $f: x \to y$ and $g: x \to z$, there exist morphisms $f': y \to w$ and $g': z \to w$ such that $f'f = g'g$. As pointed out in \cite[III.3]{MM}, the atomic topology can be defined over $\Co$ if and only if the above Ore condition is satisfied.
\end{itemize}
\end{example}

Note that the empty subfunctor $\varnothing$ can never belong to $J_d(x)$ for $x \in \Ob(\C)$. Furthermore, if $\C$ satisfies the right Ore condition, then $J_d = J_a$.

\subsection{Sheaf theory}

Fix a Grothendieck topology $J$ on $\Co$ and let $\boldsymbol{\C}^{\op} = (\Co, J)$ be the corresponding site. The following definitions are taken from \cite[III.4]{MM}.

\begin{definition} \label{def of sheaves}
Let $\Sets$ be the category of sets and maps.
\begin{itemize}
\item A \textit{presheaf of sets} on $\Co$ is a covariant functor $F: \C \to \Sets$.

\item Given $S \in J(x)$, a \textit{matching family} for $S$ of elements of $F$ is a rule assigning to each morphism $f: x \to y$ in $S$ an element $v_f \in F(y)$ such that $g \cdot v_f = v_{gf}$ for any morphism $g: y \to \bullet$, where the action of $g$ on $v_f$ is determined by the covariant functor $F$.

\item An \textit{amalgamation} of the above matching family is an element $v_x \in F(x)$ such that $f \cdot v_x = v_f$ for all $f \in S$.
\end{itemize}
\end{definition}

More formally, let $\mathscr{S}$ be the category whose objects are elements in $S$ and morphisms are dotted arrows making the following diagram commute
\[
\xymatrix{
x \ar[rr]^-{f \in S} \ar[dr]_-{g \in S} & & y\\
 & z \ar@{-->}[ur]_h
}
\]
Let $\mathscr{S}_{\ast}$ be the larger category obtained by adding the identity morphism $1_x$ on $x$ to $\mathscr{S}$. One can check that a matching family is a functor $M$ from $\mathscr{S}$ to the category of singleton sets such that $M(f) \subseteq (F \circ \cod)(f)$, where $\cod$ is the codomain functor, and an amalgamation of this matching family is an extension of $M$ from $\mathscr{S}$ to $\mathscr{S}_{\ast}$.

\begin{definition} \label{sheaf def}
A presheaf $F$ of sets on $\Co$ is called \textit{separated} if every matching family for every object $x$ and every $S \in J(x)$ has at most one amalgamation. It is a \textit{sheaf of sets} on the site $\boldsymbol{\C}^{\op} = (\Co, J)$ if every matching family has exactly one amalgamation.
\end{definition}

\begin{remark} \label{equivalent definition}
A more concise version of this definition is: $F$ is a sheaf of sets on the site $\boldsymbol{\C}^{\op}$ if and only if for every $x \in \Ob(\C)$ and every $S \in J(x)$, the inclusion $S \subseteq \C(x, -)$ induces an isomorphism
\[
\mathrm{Nat}(\C(x, -), F) \cong \mathrm{Nat}(S, F)
\]
where $\mathrm{Nat}(\bullet, F)$ means the set of natural transformations from $\bullet$ to $F$.
\end{remark}

When $J$ is the trivial topology, every presheaf of sets on $\Co$ is a sheaf of sets. For a general Grothendieck topology $J$ on $\Co$, the category $\Sh(\BC^{\op})$ of sheaves of sets on $\BC^{\op}$ is a full subcategory of $\C \lsets$, the category of all presheaves of sets on $\Co$. The inclusion functor $\iota: \Sh(\BC^{\op}) \to \C \lsets$ has a left adjoint $\sharp: \C \lsets \to \Sh(\BC^{\op})$, called the \textit{sheafification} functor. It is well known that the sheafification is a special type of localization; see \cite[Proposition 1.3]{GZ} or \cite{KS}. Therefore, $\Sh(\BC^{\op})$ is a localization of $\C \lsets$ by the multiplicative system of local isomorphisms \cite{KS} (or morphismes bicouvrants \cite{AGV}).

Now we give definitions of ringed sites and sheaves of modules on ringed sites. For more details, please refer to \cite[Chapter 18]{Stack}.

\begin{definition}
Let $\BC^{\op} = (\Co, J)$ be a site.
\begin{enumerate}
\item A \textit{presheaf of rings} on $\C^{\op}$ is a covariant functor $\mathcal{O}: \C \to \mathrm{Ring}$, to the category of associative rings.

\item A \textit{presheaf of modules} on $(\C^{\op}, \mathcal{O})$ is a covariant functor $V: \C \to \mathrm{Ab}$, to the category of abelian groups, satisfying the following conditions: the value $V_x$ of $V$ on $x$ is an $\mathcal{O}_x$-module for each $x \in \Ob(\C)$, and $V_f: V_x \to V_y$ is an $\mathcal{O}_x$-linear map for each morphism $f: x \to y$ in $\C$, where $V_y$ is viewed as an $\mathcal{O}_x$-module by the ring homomorphism $\mathcal{O}_f: \mathcal{O}_x \to \mathcal{O}_y$.

\item A presheaf $\mathcal{O}$ of rings on $\C^{\op}$ is called a \textit{sheaf of rings} on $\BC^{\op}$ if the underlying presheaf of sets is a sheaf of sets. In this case, the pair $(\BC^{\op}, \mathcal{O})$ is called a \textit{ringed site}, and $\mathcal{O}$ is called the \textit{structure sheaf}.

\item If $\mathcal{O}$ is a structure sheaf, a presheaf of modules on $(\C^{\op}, \mathcal{O})$ is called a \textit{sheaf of modules} on $(\BC^{\op}, \mathcal{O})$ if the underlying presheaf of sets is a sheaf of sets.
\end{enumerate}
\end{definition}

In this paper we also call a \textit{presheaf of modules} on $(\C^{\op}, \mathcal{O})$ an $\mathcal{O}$-module. Denote the category of all $\mathcal{O}$-modules by $\mathcal{O} \Mod$. When $\mathcal{O}$ is a structure sheaf, we denote the category of sheaves of modules on $(\BC^{\op}, \mathcal{O})$ by $\Sh(\BC^{\op}, \mathcal{O})$. Note that $\Sh(\BC^{\op}, \mathcal{O})$ is an abelian category: kernels are computed in $\mathcal{O} \Mod$; to compute cokernels, we first obtain cokernels in $\mathcal{O} \Mod$, and then apply the sheafification functor $\sharp$. Furthermore, $\Sh(\BC^{\op}, \mathcal{O})$ is a Grothendieck category, and hence every object in it has an injective hull. But in general it does not have enough projective objects.

\section{$J$-torsion theory}

In this section we develop a torsion theory induced by rules $J$ satisfying certain conditions. As before, let $\C$ be a small skeletal category, and let $\mathcal{O}$ be a presheaf of associative rings on $\Co$. We study $J$-torsion $\mathcal{O}$-modules, and describe a bijective correspondence between Grothendieck topologies on $\Co$ and certain special hereditary torsion pairs on $\mathcal{O} \Mod$. We also note that the results in this section parallel the torsion theory of module categories induced by Gabriel topologies on associative rings; for details, please see \cite{Baz} or \cite[Section 4]{BL}.

\subsection{$J$-torsion modules}

The main content of this subsection includes elementary properties of $J$-torsion $\mathcal{O}$-modules. Let $V$ be an $\mathcal{O}$-module. We define $J$-torsion elements of $V$ as follows.

\begin{definition}
For $x \in \Ob(\C)$, an element $v \in V_x$ is called \textit{$J$-torsion} if there is a certain $S \in J(x)$ such that $S \cdot v = 0$; that is, $f \cdot v = V_f(v) = 0$ for each $f \in S$.
\end{definition}

Given $S \in J(x)$, we set
\[
\mathcal{T}_S(V) = \{ v \in V_x \mid S \cdot v = 0 \},
\]
which is an $\mathcal{O}_x$-submodule of $V_x$. Indeed, given $v, w \in V_x$, $a, b \in \mathcal{O}_x$ and $f \in S$, since $f \cdot v = 0 = f \cdot w$, we have
 \[
 f \cdot (av + bw) = a (f \cdot v) + b (f \cdot w) = 0.
 \]
Thus $S \cdot (av + bw) = 0$, so $av+bw \in \mathcal{T}_S(V)$. In particular, if $\varnothing \in J(x)$, then $\mathcal{T}_{\varnothing} (V) = V_x$ by convention, so every $v \in V_x$ is trivially $J$-torsion in this case.

Now we define
\[
\mathcal{T}_J(V) = \bigoplus_{x \in \Ob(\C)} \sum_{S \in J(x)} \mathcal{T}_S(V).
\]
At this moment it is not clear that $\mathcal{T}_J(V)$ is an $\mathcal{O}$-submodule of $V$, namely a subfunctor of $V$ stable under the action of $\mathcal{O}$, or equivalently a subobject in the category $\mathcal{O} \Mod$. However, when $J$ satisfies the stability axiom, it is indeed the case.

\begin{lemma} \label{torsion submodules}
Notation as above and suppose that $J$ satisfies the stability axiom. Then:
\begin{enumerate}
\item $\mathcal{T}_J(V)$ is an $\mathcal{O}$-submodule of $V$;

\item the rule $V \to \mathcal{T}_J(V)$ is a left exact functor from $\mathcal{O} \Mod$ to itself;

\item if $J$ is closed under finite intersections, then every element in $(\mathcal{T}_J(V))_x$ is $J$-torsion for $x \in \Ob(\C)$.
\end{enumerate}
\end{lemma}

\begin{proof}
(1): It suffices to check the following condition: for each morphism $f: x \to y$ in $\C$,
\[
v \in \sum_{S \in J(x)} \mathcal{T}_S(V) \Rightarrow f \cdot v \in \sum_{S \in J(y)} \mathcal{T}_S(V).
\]

Note that every $\mathcal{T}_S(V)$ is an $\mathcal{O}_x$-submodule of $V_x$. Thus $v$ belongs to the sum of the submodules $\mathcal{T}_S(V)$, and hence can be written as $v = v_1 + v_2 + \ldots + v_s$ such that for each $v_i$ there is a certain $S_i \in J(x)$ (the special case $S_i = \varnothing$ is allowed) satisfying $S_i \cdot v_i = 0$ for $1 \leqslant i \leqslant s$. Clearly, $f^{\ast}(S_i) \cdot (f \cdot v_i) = 0$. But $f^{\ast} (S_i) \in J(y)$ by the stability axiom, so
\[
f \cdot v_i \in \mathcal{T}_{f^{\ast} (S_i)} (V) \subseteq \sum_{S \in J(y)} \mathcal{T}_S(V)
\]
for each $i$. Consequently, $f \cdot v$ is contained in $\sum_{S \in J(y)} \mathcal{T}_S(V)$ as desired.

(2): We want to show that $\mathcal{T}_J$ sends a morphism $\phi: V \to W$ between two $\mathcal{O}$-modules to a morphism $\mathcal{T}_J(V) \to \mathcal{T}_J(W)$ via restriction. It suffices to explain that $\phi_x$ sends $J$-torsion elements in $V_x$ to $J$-torsion elements in $W_x$ for $x \in \Ob(\C)$. Let $v \in V_x$ be a $J$-torsion element. If $\varnothing \in J(x)$, then every element in $W_x$ is $J$-torsion, and the conclusion holds trivially. Otherwise, there is a nonempty $S \in J(x)$ such that $f \cdot v = 0$ for every $(f: x \to y) \in S$. Consequently, $f \cdot \phi_x(v) = \phi_y (f \cdot v) = 0$, so $\phi_x(v)$ is $J$-torsion as well.

(3): Take an arbitrary $v \in (\mathcal{T}_J(V))_x$ and write it as $v_1 + \ldots + v_s$ such that $v_i \in \mathcal{T}_{S_i}(V)$ for a certain $S_i \in J(x)$. Let $S = S_1 \cap \ldots \cap S_s$ which is contained in $J(x)$ since $J$ is closed under finite intersections. Then $v \in \mathcal{T}_S(V)$, and hence is a $J$-torsion element.
\end{proof}

\textbf{Setup: In the rest of this paper we always assume that $J$ satisfies the stability axiom.}

\begin{definition}
We call $\mathcal{T}_J(V)$ the \textit{$J$-torsion submodule} of $V$. We say that $V$ is
\begin{itemize}
\item \textit{$J$-torsion} if $\mathcal{T}_J(V) = V$;

\item \textit{$J$-torsion free} if $\mathcal{T}_J(V) = 0$;

\item \textit{$J$-saturated} if $\mathrm{R}^i \mathcal{T}_J(V) = 0$ for $i \leqslant 1$, where $\mathrm{R}^i\mathcal{T}_J$ is the $i$-th right derived functor of $\mathcal{T}_J$.
\end{itemize}
\end{definition}

Denote by $\mathcal{T}(J)$ (resp., $\mathcal{F}(J)$) the category of $J$-torsion modules (resp., $J$-torsion free modules). It is clear from definitions that $\mathcal{T}(J)$ is closed under taking quotients: if $V$ lies in $\mathcal{T}(J)$ and $W$ is a submodule of $V$ in $\mathcal{O} \Mod$, then $V/W$ lies in $\mathcal{T} (J)$ as well. Dually, $\mathcal{F}(J)$ is closed under taking submodules.

Given a rule $J$, we define its \textit{inclusion closure} $\tilde{J}$ as follows: for every $x \in \Ob(\C)$,
\[
\tilde{J} (x) = \{ S \subseteq \C(x, -) \mid \, \exists \, T \in J(x) \text{ such that } T \subseteq S \}.
\]
Clearly, $J$ is closed under upward inclusions if and only if $\tilde{J} = J$.

\begin{lemma} \label{inclusion closure}
The inclusion closure $\tilde{J}$ satisfies the stability axiom, and $\mathcal{T}(J) = \mathcal{T}(\tilde{J})$.
\end{lemma}

\begin{proof}
For the first statement, we take an arbitrary $x \in \Ob(\C)$, an arbitrary $S \in \tilde{J}(x)$, and an arbitrary morphism $f: x \to y$ in $\C$. By definition, there is a certain $T \in J(x)$ such that $T \subseteq S$. Then one has $f^{\ast} (S) \supseteq f^{\ast} (T)$. But $f^{\ast} (T) \in J(y)$ since $J$ satisfies the stability axiom. By the definition of $\tilde{J}$, one shall have $f^{\ast}(S) \in \tilde{J}(y)$, so $\tilde{J}$ also satisfies the stability axiom.

To establish the second statement, one notes that $J$-torsion modules are generated by $J$-torsion elements. Therefore, it suffices to show that $J$-torsion elements coincide with $\tilde{J}$-torsion elements for any $\mathcal{O}$-module $V$. Since every $S \in J(x)$ is clearly contained in $\tilde{J}(x)$ for $x \in \Ob(\C)$, $J$-torsion elements are clearly $\tilde{J}$-torsion elements. On the other hand, if an element $v$ in a certain $V_x$ is a $\tilde{J}$-torsion element, then one can find some $S \in \tilde{J}(x)$ such that $S \cdot v = 0$. But by the definition of $\tilde{J}$, there exists some $T \in J(x)$ such that $T \subseteq S$, and in particular $T \cdot v = 0$, so $v$ is also $J$-torsion.
\end{proof}

\begin{remark}
By this lemma, when considering $J$-torsion modules, without loss of generality we can always assume that $J$ is closed under upward inclusions.
\end{remark}

When $\mathcal{O}$ is a constant structure sheaf given by a ring which is still denoted by $\mathcal{O}$, for $S \in J(x)$, we have a short exact sequence of $\mathcal{O}$-modules
\[
0 \to \underline{S} \to P(x) \to P(x)/\underline{S} \to 0
\]
where $P(x)$ and $\underline{S}$ are the $\mathcal{O}$-linearizations of $\C(x, -)$ and $S$, respectively. We need to generalize this construction to all structure sheaves as the quotient $\mathcal{O}$-module in the above sequence will play a key role for us to prove many results.

Denote by $\C_x$ the subcategory with the single object $x$ and the identity morphism $1_x$. The inclusion functor $\iota_x: \C_x \to \C$ induces a restriction functor $\res: \mathcal{O} \Mod \to \mathcal{O}_x \Mod$ and its left adjoint $\ind: \mathcal{O}_x \Mod \to \mathcal{O} \Mod$ is given by the left Kan extension. Explicitly, given an $\mathcal{O}_x$-module $W$ as well as an object $y$ in $\C$, one has
\[
(\ind (W))_y = \varinjlim_{f: \, x \to y} (\mathcal{O}_y \otimes_{\mathcal{O}_x} W) = \bigoplus_{f: \, x \to y} (\mathcal{O}_y \otimes_{\mathcal{O}_x} W),
\]
where $\mathcal{O}_y$ is regarded as a right $\mathcal{O}_x$-module for each $f$ via the ring homomorphism $\mathcal{O}_f: \mathcal{O}_x \to \mathcal{O}_y$. For more details, please refer to \cite[Lemma 2.2 and Proposition 2.4]{DLLY}. In particular, to deduce the second equality, we note that the colimit is taken over the category whose objects are morphisms $f: x \to y$ in $\C$ and whose morphisms are morphisms $h$ in $\C_x$ making the following diagram commute
\[
\xymatrix{
x \ar[rr]^-{f} \ar@{-->}[dr]_-{h} & & y.\\
 & x \ar[ur]_{f'}
}
\]
It is easy to see that this is a discrete category since $\C_x$ only has the single morphism $1_x$, and consequently the colimit is actually a coproduct; see \cite[Example 2.3]{DLLY}.

Now let $P(x) = \ind (\mathcal{O}_x)$, and $\underline{S}$ be the submodule such that
\[
\underline{S}_y = \bigoplus_{(f: \, x \to y) \in S} (\mathcal{O}_y \otimes_{\mathcal{O}_x} \mathcal{O}_x) \cong \bigoplus_{(f: \, x \to y) \in S} \mathcal{O}_{y, f} \subseteq \bigoplus_{f \in \C(x, y)} \mathcal{O}_{y, f} = P(x)_y
\]
where $\mathcal{O}_{y, f} = \mathcal{O}_y$ for each $f$. It is easy to see that $P(x)/\underline{S} = 0$ if and only if $S = \C(x, -)$. In the case that $S \neq \C(x, -)$, the value $(P(x) / \underline{S})_x$ of $P(x)/\underline{S}$ on $x$ is a nonzero free  $\mathcal{O}_x$-module since it contains a direct summand $\mathcal{O}_{x, 1_x}$. Denote the identity in this direct summand by $\boldsymbol{1}_x$.

\begin{lemma} \label{correspondence}
Given two rules $J$ and $K$, $\mathcal{T}(K)$ is a full subcategory of $\mathcal{T}(J)$ if and only if $K \subseteq \tilde{J}$.
\end{lemma}

\begin{proof}
If $K \subseteq \tilde{J}$, then by the definition $\mathcal{T}(K)$ is a full subcategory of $\mathcal{T}(\tilde{J})$, which coincides with $\mathcal{T}(J)$ by Lemma \ref{inclusion closure}.

Now we prove the only if direction. That is, for an arbitrary $x \in \Ob(\C)$ and $S \in K(x)$, we want to show that $S \in \tilde{J}(x)$. The conclusion holds trivially for $S = \C(x, -)$, so we assume that $S$ is not $\C(x, -)$. Note that a morphism $f: x \to y$ in $S$ induces a homomorphism of $\mathcal{O}_x$-modules from $\mathcal{O}_{x, 1_x} \subseteq P(x)_x$ to $\mathcal{O}_{y, f} \subseteq P(x)_y$. Passing to the quotient, we conclude that the identity $\boldsymbol{1}_x$ of $\mathcal{O}_{x, 1_x}$ in $(P(x) / \underline{S})_x$ is a $K$-torsion element since $f \cdot \boldsymbol{1}_x = 0$ for any $f \in S$. Thus $P(x) / \underline{S}$ is a $K$-torsion module since it is generated by $\boldsymbol{1}_x$ as an $\mathcal{O}$-module, so it is also a $J$-torsion module as $\mathcal{T}(K)$ is a full subcategory of $\mathcal{T}(J)$.

Although $\boldsymbol{1}_x$ might not be a $J$-torsion element, we can write $\boldsymbol{1}_x = e_1 + \ldots + e_r$ where each $e_i$ is a nonzero $J$-torsion element. In particular, for $i=1$, we can find a certain $S_1 \in J(x)$ such that $f \cdot e_1 = 0$ in $P(x) / \underline{S}$ for each $(f: x \to \bullet) \in S_1$. Since $f$ sends $e_1$ to $\boldsymbol{1}_{\bullet} \otimes e_1 \in \mathcal{O}_{\bullet, f} \subseteq P(x)_{\bullet}$, the image of $\boldsymbol{1}_{\bullet} \otimes e_1$ in the quotient $(P(x) / \underline{S})_{\bullet}$ is 0 if and only if $f \in S$. Thus we have $f \in S$, and hence $S_1 \subseteq S$. It follows that $S \in \tilde{J}(x)$ as desired.
\end{proof}

\begin{remark}
We shall point out that the above proof works for the special case that $S = \varnothing$. In this situation, $\underline{S}$ is the zero module, so $P(x) = P(x) / \underline{S}$, and $S_1 = \varnothing = S$.
\end{remark}

\begin{proposition} \label{bijective correspondence}
Let $J$ and $K$ be two rules. Then $\tilde{J} = \tilde{K}$ if and only if $\mathcal{T}(J) = \mathcal{T}(K)$. In particular, if $J$ and $K$ are Grothendieck topologies on $\Co$, then $J = K$ if and only if $\mathcal{T}(J) = \mathcal{T}(K)$.
\end{proposition}

\begin{proof}
The only if direction follows from Lemma \ref{inclusion closure}. For the other direction, note that by Lemma \ref{correspondence} one shall have $J \subseteq \tilde{K}$ and $K \subseteq \tilde{J}$, so $\tilde{J} \subseteq \tilde{K}$ and $\tilde{K} \subseteq \tilde{J}$.
\end{proof}

We have described a procedure to construct a full subcategory $\mathcal{T}(J)$ of $\mathcal{O} \Mod$ for a rule $J$. Now let us consider an inverse procedure.

\begin{definition}
Given an $\mathcal{O}$-module $V$, an object $x \in \Ob(\C)$, and an element $v \in V_x$, we define the \textit{annihilator} of $v$ to be
\[
\mathcal{A}(v) = \{ f: x \to \bullet \mid f \cdot v = 0 \},
\]
which is clearly a left ideal of morphisms on $x$. For a nonempty full subcategory $\mathfrak{A}$ of $\mathcal{O} \Mod$, we define $\mathcal{A} (\mathfrak{A})$ to be the rule assigning to each $x \in \Ob(\C)$ the collection
\[
\mathcal{A} (\mathfrak{A}) (x) = \{ \mathcal{A}(v) \mid v \in V_x, \, V \in \Ob(\mathfrak{A}) \}.
\]
\end{definition}

Clearly, $\mathcal{A} (\mathfrak{A})$ satisfies the maximal axiom since $\mathcal{A}(0) = \C(x, -)$ for $0 \in V_x$. It also satisfies the stability axiom.

\begin{lemma}
The rule $\mathcal{A} (\mathfrak{A})$ satisfies the stability axiom.
\end{lemma}

\begin{proof}
Take $S \in \mathcal{A} (\mathfrak{A})(x)$ and an arbitrary morphism $f: x \to y$. If $S \neq \varnothing$, then there is an $\mathcal{O}$-module $V$ in $\mathfrak{A}$ as well as an element $v \in V_x$ such that $S = \mathcal{A}(v)$. For $f \cdot v \in V_y$, we have
\[
f^{\ast} (S) = \{ g: y \to \bullet \mid gf \in S  \} = \{ g: y \to \bullet \mid g \cdot (f \cdot v) = 0 \} = \mathcal{A}(f \cdot v).
\]
Thus $f^{\ast} (S) \in \mathcal{A} (\mathfrak{A})(y)$. For $S = \varnothing$, we can find an element $v \in V_x$ such that $f \cdot v \neq 0$ for all $f \in \C(x, -)$. In particular, for a fixed $f: x \to y$, we have
\[
\varnothing = f^{\ast} (S) = \{g: y \to \bullet \mid g \cdot (f \cdot v) = (gf) \cdot v = 0 \} = \mathcal{A}(f \cdot v) \in \mathcal{A} (\mathfrak{A})(y),
\]
so the stability axiom also holds for this case.
\end{proof}

The following result can be viewed as an analogue of the well known correspondence between algebraic sets and radical ideals given by Hilbert's Nullstellensatz. In particular, $\mathcal{T}$ is the analogue of the operation taking algebraic sets, and $\mathcal{A}$ is the analogue of the operation taking ideals associated to algebraic sets. With respect to these analogues, the full subcategory of $J$-torsion $\mathcal{O}$-modules, where $J$ is closed under finite intersections, plays the role of an algebraic set, and the rule $J$ closed under finite intersections and upward inclusions plays the role of a radical ideal.

\begin{proposition} \label{A recovers J}
Given a rule $J$, one has:
\begin{enumerate}
\item $\mathcal{T} ( \mathcal{A} (\mathcal{T} (J))) = \mathcal{T}(J)$ if $J$ is closed under finite intersections;

\item $J = \mathcal{A}(\mathcal{T}(J))$ if and only if $J$ is closed under finite intersections and upward inclusions.
\end{enumerate}
\end{proposition}

\begin{proof}
For brevity denote $\mathcal{A} (\mathcal{T}(J))$ by $K$.

(1): For $x \in \Ob(\C)$ and $S \in J(x)$, consider the $\mathcal{O}$-module $P(x) / \underline{S}$ defined before Lemma \ref{correspondence}. It is a $J$-torsion module, and hence is contained in $\mathcal{T}(J)$. By the definition of $K$, we conclude that $S = \mathcal{A}(\boldsymbol{1}_x) \in K(x)$, so $J \subseteq K$. Consequently, $\mathcal{T}(J)$ is a full subcategory of $\mathcal{T}(K)$. Now we prove $\mathcal{T}(K)\subseteq \mathcal{T}(J)$ by showing that every $K$-torsion $\mathcal{O}$-module $V$ is also $J$-torsion. Since torsion modules are generated by torsion elements, it suffices to show the following claim: for $x \in \Ob(\C)$ and $v \in V_x$, if $v$ is $K$-torsion, then it is $J$-torsion.

Assuming that $v$ is $K$-torsion, we can find a certain $S \in K(x)$ such that $S \cdot v = 0$. By definition of $K$, we can find a $J$-torsion $\mathcal{O}$-module $W$ and an element $w \in W_x$ such that $S = \mathcal{A}(w)$. But $J$ is closed under finite intersections, so $w$ is a $J$-torsion element by Lemma \ref{torsion submodules}, and hence there is a certain $S' \in J(x)$ such that $S' \cdot w = 0$. Clearly, we shall have $S' \subseteq S = \mathcal{A}(w)$, so $S' \cdot v = 0$. Consequently, $v$ is a $J$-torsion element as well, and the claim is established.

(2): We have shown that $J \subseteq K$ in the proof of (1), so it suffices to prove $K \subseteq J$ to establish the if direction. The given assumption tells us that $J = \tilde{J}$. Thus by Lemma \ref{correspondence}, it suffices to show that $\mathcal{T} (K)$ is a full subcategory of $\mathcal{T}(J)$. But this is true by (1).

Now we turn to the only if direction. One first checks that $J = K$ is closed under finite intersections. For $x \in \Ob(\C)$ and $S, S' \in K(x)$, we can find objects $V$ and $V'$ in $\mathcal{T}(J)$ and elements $v \in V_x$ and $v' \in V'_x$ such that $S = \mathcal{A}(v)$ and $S' = \mathcal{A}(v')$. Note that $\mathcal{T}(J)$, though might not be closed under extensions, is closed under finite direct sums. Consequently, $V \oplus V'$ is contained in $\mathcal{T}(J)$. In particular, for $(v, v') \in V_x \oplus V_x'$, one has $S \cap S' = \mathcal{A}((v, v')) \in K(x)$, so $K$ is closed under finite intersections.

Finally we check that $J$ is closed under upward inclusions. Suppose that $S \subseteq T \subseteq \C(x, -)$ and $S \in J(x)$. Consider $\mathcal{O}$-modules $P(x) / \underline{S}$ and $P(x)/\underline{T}$. Note that $P(x) / \underline{S}$ is contained in $\mathcal{T}(J)$, $P(x)/\underline{T}$ is a quotient of $P(x)/\underline{S}$, and $\mathcal{T}(J)$ is closed under taking quotients. Consequently, $P(x)/\underline{T}$ is also contained in $\mathcal{T}(J)$. Thus $T = \mathcal{A} (\boldsymbol{1}_x) \in K(x) = J(x)$, so $J$ is closed under upward inclusions.
\end{proof}

The above proof shows that $\mathcal{A} (\mathcal{T}(J))$ is always closed under finite intersections.

\begin{remark}
We illustrate the above result from a categorical viewpoint. Let $\mathcal{J}$ be the set of rules $J$ satisfying the stability axiom. It is a poset under the inclusion. We can construct another category whose objects are full subcategories of $\mathcal{O} \Mod$ and morphisms are inclusion functors. Then $\mathcal{A}$ and $\mathcal{T}$ are covariant functors between these two categories. Restricting to the subposet consisting of $J$ closed under finite intersections, we have $\mathcal{T} \mathcal{A} \mathcal{T} = \mathcal{T}$; and restricting further to the subposet consisting of $J$ closed under finite intersections and upward inclusions, $\mathcal{A} \mathcal{T}$ is the identity functor.

It is not hard to see that $\mathcal{T}$ preserves finite intersections. Explicitly, given two rules $J$ and $K$ closed under upward inclusions and finite intersections, one has $\mathcal{T}(J) \cap \mathcal{T}(K) = \mathcal{T}(J \cap K)$, where $J \cap K$ is the rule assigning $J(x) \cap K(x)$ to $x \in \Ob(\C)$. Indeed, since $\mathcal{T}(J \cap K) \subseteq \mathcal{T}(J)$ and $\mathcal{T}(J \cap K) \subseteq \mathcal{T}(K)$, one has $\mathcal{T}(J) \cap \mathcal{T}(K) \supseteq \mathcal{T}(J \cap K)$. On the other hand, take $V$ in $\mathcal{T}(J) \cap \mathcal{T}(K)$. Then for $x \in \Ob(\C)$ and $v \in V_x$, one can find $S \in J(x)$ and $T \in K(x)$ such that $S \cdot v  = 0 = T \cdot v$. Then $S \cup T \in J(x) \cap K(x)$ and $(S \cup T) \cdot v = 0$, so $v$ is also $J \cap K$-torsion, and hence $V \in \mathcal{T}(J \cap K)$.
\end{remark}

\subsection{Torsion pairs induced by Grothendieck topologies}

In this subsection we show that Grothendieck topologies induce hereditary torsion pairs in $\mathcal{O} \Mod$, and are completely determined by the induced torsion pairs.

Recall that a torsion pair in an abelian category $\mathscr{A}$ is a pair $(\mathcal{T}, \mathcal{F})$ of full subcategories such that
\[
\Hom_{\mathscr{A}}(T, F) = 0
\]
for all $T \in \mathcal{T}$ and $F \in \mathcal{F}$, and every object $M \in \mathscr{A}$ admits a short exact sequence
\[
0 \longrightarrow T \longrightarrow M \longrightarrow F \longrightarrow 0
\]
with $T \in \mathcal{T}$ and $F \in \mathcal{F}$. The objects of $\mathcal{T}$ are called \emph{torsion} objects and those of $\mathcal{F}$ \emph{torsion-free} objects. A torsion pair $(\mathcal{T}, \mathcal{F})$ is called \emph{hereditary} if $\mathcal{T}$ is closed under taking subobjects. Equivalently, $\mathcal{T}$ is
closed under subobjects, quotients, and extensions, and every object $M$ has a largest torsion subobject $\mathcal{T}(M) \subseteq M$ such that
$M/\mathcal{T}(M)$ is torsion-free.

The following proposition gives the precise correspondence between hereditary torsion pairs arising from Grothendieck topologies.

\begin{proposition} \label{torsion theory}

Given a rule $J$, one has:
\begin{enumerate}
\item $(\mathcal{T}(J), \, \mathcal{F}(J))$ is a hereditary torsion pair in $\mathcal{O} \Mod$ if $J$ is a Grothendieck topology.

\item Conversely, if $J$ is closed under upward inclusions and $(\mathcal{T}(J), \, \mathcal{F}(J))$ is a hereditary torsion pair in $\mathcal{O} \Mod$, then $J$ is a Grothendieck topology.
\end{enumerate}
\end{proposition}

\begin{proof}
(1): Firstly, $\Hom_{\mathcal{O} \Mod} (T, F) = 0$ for $T$ in $\mathcal{T}(J)$ and $F$ in $\mathcal{F}(J)$. Indeed, take any morphism $\phi: T \to F$. By (2) of Lemma \ref{torsion submodules}, $\phi$ restricts to a morphism
\[
\mathcal{T}_J(T) = T \longrightarrow \mathcal{T}_J(F) = 0,
\]
which is the zero morphism. Therefore, $\phi = 0$. Secondly, it is clear that $\mathcal{T}(J)$ is closed under taking quotient modules. It is also closed under taking submodules. Indeed, for any $V$ in $\mathcal{T}(J)$, $x \in \Ob(\C)$ and $v \in V_x$, $v$ is a $J$-torsion element by Lemma \ref{torsion submodules}. Consequently, for an arbitrary submodule $W \subseteq V$ and every $x \in \Ob(\C)$, all elements in $W_x$ are also $J$-torsion.

Clearly, only the zero module lies in the intersection of $\mathcal{T}(J)$ and $\mathcal{F}(J)$. To deduce the conclusion, it suffices to show that $V/\mathcal{T}_J(V)$ is $J$-torsion free for every $\mathcal{O}$-module $V$; that is, for $x \in \Ob(\C)$ and $\bar{v} \in (V/\mathcal{T}_J(V))_x$, if $\bar{v}$ is $J$-torsion, then $\bar{v} = 0$.

If $\varnothing \in J(x)$, then every $v \in V_x$ is $J$-torsion, so the conclusion holds trivially as $(V/\mathcal{T}_J(V))_x = 0$. Otherwise, take a preimage $v \in V_x$ of $\bar{v}$. Since $\bar{v}$ is $J$-torsion, there exists a nonempty $S \in J(x)$ such that $f \cdot v \in (\mathcal{T}_J(V))_{\bullet}$ for each $f: x \to \bullet$ in $S$. But $f \cdot v$ is $J$-torsion by Lemma \ref{torsion submodules}, so we can find a certain $S_f \in J(\bullet)$ such that $S_f \cdot (f \cdot v) = 0$. Define
\[
\hat{S} = \bigcup_{f \in S} S_f f = \bigcup_{f \in S} \{gf \mid g \in S_f \},
\]
which is clearly a subfunctor of $\C(x, -)$ and $\hat{S} \cdot v = 0$. Since for every $f \in S$,
\[
f^{\ast}(\hat{S}) \supseteq f^{\ast}(S_f f) \supseteq S_f \in J(\bullet),
\]
one has $f^{\ast} (\hat{S}) \in J(\bullet)$ because $J$ is closed under upward inclusions, and hence $\hat{S} \in J(x)$ by the transitivity axiom. Consequently, $v$ is $J$-torsion and is contained in $(\mathcal{T}_J(V))_x$, so $\bar{v} = 0$ as desired.

(2) The maximal axiom holds trivially as $J$ is closed under upward inclusions. Since we have assumed that $J$ always satisfies the stability axiom, it remains to check the transitivity axiom. Given $T \in J(x)$ and $S \subseteq \C(x, -)$ such that $g^{\ast}(S) \in J(y)$ for every $g: x \to y$ in $T$, we need to show that $S \in J(x)$. Of course, we can assume $S \neq \C(x, -)$.

The torsion pair gives a short exact sequence
\[
0 \to P' \to P(x) / \underline{S} \to P'' \to 0
\]
where $P(x) / \underline{S}$ is constructed before Lemma \ref{correspondence}, $P'$ is $J$-torsion and $P''$ is $J$-torsion free. Consider $\boldsymbol{1}_x \in \mathcal{O}_{x, 1_x} \subseteq (P(x) / \underline{S})_x$ which is specified before Lemma \ref{correspondence}. For every $g: x \to y$ in $T$, since $g^{\ast} (S) \in J(y)$ and $g^{\ast} (S) \cdot (g \cdot \boldsymbol{1}_x) = 0$, it follows that $g \cdot \boldsymbol{1}_x$ is $J$-torsion, and hence is contained in $P'_y$. Passing to the quotient module $P''$, we deduce that $T \cdot \bar{\boldsymbol{1}}_x = 0$, where $\bar{\boldsymbol{1}}_x$ is the image of $\boldsymbol{1}_x$ in $P_x''$. Thus $\bar{\boldsymbol{1}}_x$ is $J$-torsion. But $P''$ is torsion free, this forces $\bar{\boldsymbol{1}}_x = 0$, or equivalently $\boldsymbol{1}_x$ is contained in $P'_x$. But $P(x) / \underline{S}$ as an $\mathcal{O}$-module is generated by $\boldsymbol{1}_x$, this forces $P(x) / \underline{S} = P'$; that is, $P(x) / \underline{S}$ is a $J$-torsion module. Now we can apply the same argument as in the last paragraph of the proof of Lemma \ref{correspondence} to deduce $S \in J(x)$.
\end{proof}

\begin{remark}
If $(\mathcal{T}(J), \, \mathcal{F}(J))$ is a hereditary torsion pair, then $\mathcal{T}(J)$ is a Serre subcategory of $\mathcal{O} \Mod$, so we can form a Serre quotient $\mathcal{O} \Mod / \mathcal{T}(J)$. In the next section we will show that it is equivalent to $\Sh(\BC^{\op}, \mathcal{O})$ when $\mathcal{O}$ is a structure sheaf.
\end{remark}

\begin{remark}
The condition that $J$ is closed under upward inclusions in part (2) of the previous proposition cannot be dropped. For instance, take $J$ to be the maximal topology on $\Co$, and let $K$ be the rule such that $K(x) = \{ \varnothing, \, \C(x, -) \}$ for each $x \in \Ob(\C)$. It is easy to see that $K$ satisfies the stability axiom and the maximal axiom. Furthermore, $J$ and $K$ induce the same hereditary torsion pair on $\mathcal{O} \Mod$ with $\mathcal{T}(J) = \mathcal{T}(K) = \mathcal{O} \Mod$ and $\mathcal{F}(J) = \mathcal{F}(K) = \{ 0 \}$. But in general $K$ is not a Grothendieck topology on $\Co$.
\end{remark}

We end this section by an example illustrating our constructions.

\begin{example}
Let $\C$ be the free category of the following quiver
\[
\xymatrix{
x \ar@/^/[rr]^f \ar@/_/[rr]_g & & y.
}
\]
By a direct computation, we obtain all possible Grothendieck topologies on $\Co$, listed below:

\begin{center}
\begin{tabular}{c c c}
  Topologies & $J(x)$ & $J(y)$ \\
  \hline
  Trivial topology & $\{ \{1_x, \, f, \, g \} \}$ & $\{ \{1_y\} \}$ \\
  \hline
  Dense topology & $\{ \{1_x, \, f, \, g \}, \, \{f, \, g \} \}$ & $\{ \{1_y\} \}$ \\
  \hline
  Maximal topology & $\{ \{1_x, \, f, \, g \}, \{f, \, g \}, \, \{f\}, \, \{g\}, \, \varnothing \}$ & $\{ \{1_y\}, \, \varnothing \}$ \\
  \hline
  Topology IV & $\{ \{1_x, \, f, \, g \}\}$ & $\{ \{1_y\}, \, \varnothing \}$. \\
  \hline
\end{tabular}
\end{center}
The corresponding torsion pairs induced by these Grothendieck topologies are:
\begin{center}
\begin{tabular}{c c c}
  Topologies & $\mathcal{T}(J)$ & $\mathcal{F}(J)$ \\
  \hline
  Trivial topology & $0$ & $\mathcal{O} \Mod$ \\
  \hline
  Dense topology & $\{ V \in \mathcal{O} \Mod \mid V_y = 0 \}$ & $\{ V \in \mathcal{O} \Mod \mid \ker V_f \cap \ker V_g = 0 \}$ \\
  \hline
  Maximal topology & $\mathcal{O} \Mod$ & $0$ \\
  \hline
  Topology IV & $\{ V \in \mathcal{O} \Mod \mid V_x = 0 \}$ & $\{ V \in \mathcal{O} \Mod \mid V_y = 0 \}$. \\
  \hline
\end{tabular}
\end{center}
\end{example}

\smallskip

\section{A torsion theoretic characterization of sheaves}

In this section we interpret notions in sheaf theory in terms of the torsion theory developed in the previous section. We mention that all results have been proved in \cite{DLLX} for the special case that $J$ is the atomic Grothendieck topology and $\mathcal{O}$ is a constant structure sheaf. Furthermore, the strategy of proofs in that paper extends to the general framework in this paper with suitable modifications.

Throughout this section let $J$ be a Grothendieck topology on $\Co$ and let $\mathcal{O}$ be a structure sheaf on the site $(\Co, J)$. Recall that an $\mathcal{O}$-module $V$ is \textit{right perpendicular} to all $J$-torsion $\mathcal{O}$-modules in the sense of Geigle-Lenzing \cite{GLen} if $\Hom_{\mathcal{O} \Mod} (T, V) = 0 = \Ext_{\mathcal{O} \Mod}^1 (T, V)$ for each $J$-torsion module $T$.

\begin{lemma} \label{preliminary equivalences}
Let $V$ be an $\mathcal{O}$-module. The following statements are equivalent:
\begin{enumerate}
\item $V$ is a $J$-torsion free $\mathcal{O}$-module;

\item $V$ is a separated presheaf of modules;

\item $\Hom_{\mathcal{O} \Mod} (T, V) = 0$ for $T \in \mathcal{T}(J)$.
\end{enumerate}
\end{lemma}

\begin{proof}
It is clear that $V$ is $J$-torsion free if and only if (3) holds, so we only need to show the equivalence of (1) and (2). If $V$ is not $J$-torsion free, then there exist $x \in \Ob(\C)$, $0 \neq v \in V_x$, and $S \in J(x)$ such that $S \cdot v = 0$. Now by setting $v_f = 0 \in V_{\bullet}$ for each $(f: x \to \bullet) \in S$, we get a matching family which has at least two amalgamations, namely $0 \in V_x$ and $v \in V_x$. Thus $V$ is not separated. Conversely, if $V$ is not separated, we can find $x \in \Ob(\C)$, $S \in J(x)$, and a matching family $\{ v_f \in V_{\bullet} \mid (f: x \to \bullet) \in S\}$ with two distinct amalgamations $v, w \in V_x$. But in this case $v - w$ is sent to 0 by any $f \in S$, so $V$ is not $J$-torsion free. This finishes the proof.
\end{proof}

Now we are ready to describe homological characterizations of sheaves of modules on $(\BC^{\op}, \mathcal{O})$.

\begin{theorem} \label{characterize sheaves}
Let $V$ be an $\mathcal{O}$-module. The following statements are equivalent:
\begin{enumerate}
\item $V$ is a sheaf of modules on $(\BC^{\op}, \mathcal{O})$;

\item $V$ is a $J$-saturated $\mathcal{O}$-module;

\item $V$ is right perpendicular to all $J$-torsion modules.
\end{enumerate}
\end{theorem}

\begin{proof}
By Lemma \ref{preliminary equivalences}, to prove the theorem, we can assume that $V$ is a $J$-torsion free $\mathcal{O}$-module.

$(2) \Leftrightarrow (3)$. In \cite[Lemma 3.6]{DLLX} we have proved the conclusion for the special case that $J$ is the atomic Grothendieck topology. We use the same strategy to establish the equivalence for all Grothendieck topologies.

Suppose that $V$ is $J$-torsion free. It has an injective hull $E$ in $\mathcal{O} \Mod$ since $\mathcal{O} \Mod$ is a Grothendieck category. We claim that $E$ is $J$-torsion free. Otherwise, we have $\mathcal{T}_J(E) \neq 0$. But $V \cap \mathcal{T}_J(E) = 0$ since $V$ is $J$-torsion free. This is impossible because the inclusion $V \to E$ is essential by \cite[Definition 47.2.1]{Stack}. Thus $E$ is $J$-torsion free and injective. Applying $\Hom_{\mathcal{O} \Mod} (T, -)$ to the short exact sequence $0 \to V \to E \to E/V \to 0$ of $\mathcal{O}$-modules we deduce that $V$ is right perpendicular to every $J$-torsion module $T$ if and only if $E/V$ is $J$-torsion free. But this is equivalent to the condition that $V$ is $J$-saturated via applying the functor $\mathcal{T}_J$ to the short exact sequence.

$(1) \Rightarrow (2)$. Let $V$ be a sheaf of modules on $(\BC^{\op}, \mathcal{O})$. It is $J$-torsion free. Since the category of $J$-torsion modules is a localizing subcategory by \cite[Proposition 2.5]{GLen}, applying \cite[Proposition 2.2]{GLen} we obtain a short exact sequence
\[
0 \to V \to \tilde{V} \to W \to 0
\]
such that $W$ is $J$-torsion and $\tilde{V}$ is right perpendicular to all $J$-torsion modules, and hence $J$-saturated by the equivalence of (2) and (3). To show that $V$ is $J$-saturated, it suffices to check that $W$ is $J$-torsion free, which forces $W = 0$ and hence $V \cong \tilde{V}$.

Let $w \in W_x$ be a $J$-torsion element for some $x \in \Ob(\C)$. We can find a certain $S \in J(x)$ such that $S \cdot w = 0$. Let $\tilde{w} \in \tilde{V}_x$ be a preimage of $w$, so we have $f \cdot \tilde{w} \in V_y$ for every $f: x \to y$ in $S$. It is clear that the rule
\[
(f: x \to y) \in S \mapsto f \cdot \tilde{w}
\]
is a matching family for $S$ of elements in $V$. Since $V$ is a sheaf of modules, we can find an amalgamation $\tilde{u} \in V_x$ such that $f \cdot \tilde{u} = f \cdot \tilde{w}$ for every $f \in S$. We can view $\tilde{u}$ as an element in $\tilde{V}_x$ by identifying $V$ with a submodule of $\tilde{V}$. Thus in the $\mathcal{O}$-module $\tilde{V}$ one has $f \cdot (\tilde{u} - \tilde{w}) = 0$ for $f \in S$. But $\tilde{V}$ is $J$-torsion free, so $\tilde{w} = \tilde{u} \in V_x$, and hence $w = 0$. Consequently, $W$ is $J$-torsion free.

$(2) \Rightarrow (1)$: Suppose that $V$ is $J$-saturated. It is a separated presheaf of modules by Lemma \ref{preliminary equivalences}. Take an arbitrary $x \in \Ob(\C)$, a certain $S \in J(x)$, and a matching family for $S$ of elements in $V$
\[
\{v_f \in V_y \mid (f: x \to y) \in S \}. \quad \quad (\dag)
\]
 We need to show the existence of amalgamations. By \cite[III.5, Lemma 2]{MM}, the canonical map $\eta_V: V \to V^+$ is injective where $V^+ = V^{\sharp}$ is the sheafification of $V$. Thus we obtain the following short exact sequence in $\mathcal{O} \Mod$:
\[
0 \to V \to V^+ \to W \to 0.
\]
Furthermore, by \cite[III.5, Lemma 5]{MM}, $V^+$ is a sheaf of modules on $(\BC^{\op}, \mathcal{O})$, and hence a $J$-saturated $\mathcal{O}$-module by the conclusion of the only if direction. Applying $\mathcal{T}_J$ to this short exact sequence, we deduce that $W$ is $J$-torsion free since $V$ is $J$-saturated.

We identify $V$ as a submodule of $V^+$ via the canonical map $\eta_V$. Under this identification, the matching family $(\dag)$ is a matching family for $S$ of elements in $V^+$. But $V^+$ is a sheaf of modules on $(\BC^{\op}, \mathcal{O})$, so we can find an amalgamation $v \in V^+_x$ such that $f \cdot v = v_f$ for $f \in S$. Passing to the quotient $W$, this means $f \cdot \bar{v} = \overline{v_f}$ for $f \in S$, where $\bar{v} \in W_x$ and $\overline{v_f} \in W_y$ are the images of $v \in V^+_x$ and $v_f \in V^+_y$ respectively. But $\overline{v_f} = 0$ as $v_f \in V_y$, so $\bar{v}$ is $J$-torsion. Since $W$ is $J$-torsion free, this happens if and only if $\bar{v} = 0$, namely $v \in V_x$ is the desired amalgamation.
\end{proof}

We describe a few consequences of the above theorem. The first one can be viewed as a ringed version of Remark \ref{equivalent definition}.

\begin{corollary} \label{another equivalence}
Given an $\mathcal{O}$-module $V$, the following statements are equivalent:
\begin{enumerate}
\item $V$is a sheaf of modules on $(\BC^{\op}, \mathcal{O})$;

\item for any $x \in \Ob(\C)$ and any $S \in J(x)$, the inclusion $\underline{S} \to P(x)$ induces an isomorphism
\[
\Hom_{\mathcal{O} \Mod} (P(x), V) \cong \Hom_{\mathcal{O} \Mod} (\underline{S}, V),
\]
where $P(x)$ and $\underline{S}$ are as defined before Lemma \ref{correspondence};

\item for any $x \in \Ob(\C)$ and any $S \in J(x)$, the inclusion $\mathbb{Z}S \to \mathbb{Z}\C(x, -)$ induces an isomorphism
\[
\Hom_{\underline{\mathbb{Z}} \Mod} (\mathbb{Z}\C(x, -), V) \cong \Hom_{\underline{\mathbb{Z}} \Mod} (\mathbb{Z}S, V),
\]
where $\underline{\mathbb{Z}}$ is the constant structure sheaf induced by the ring of integers and $\mathbb{Z}S$ is the free $\mathbb{Z}$-module spanned by elements in $S$.
\end{enumerate}
\end{corollary}

\begin{proof}
(1)$\Rightarrow$(2). Applying $\Hom_{\mathcal{O} \Mod} (-, V)$ to the short exact sequence
\[
0 \to \underline{S} \to P(x) \to P(x) / \underline{S} \to 0
\]
we obtain another exact sequence
\begin{align*}
& 0 \to \Hom_{\mathcal{O} \Mod} (P(x) / \underline{S}, V) \to \Hom_{\mathcal{O} \Mod} (P(x), V)\\
& \to \Hom_{\mathcal{O} \Mod} (\underline{S}, V) \to \Ext^1_{\mathcal{O} \Mod} (P(x) / \underline{S}, V) \to 0
\end{align*}
by noting that $P(x)$ is projective in $\mathcal{O} \Mod$. Indeed, by the construction before Lemma \ref{correspondence}, we get a pair of adjoint functors $(\mathrm{ind}, \mathrm{res})$ between $\mathcal{O} \Mod$ and $\mathcal{O}_x \Mod$. Since the functor $\mathrm{res}$ is exact, the functor $\mathrm{ind}$ preserves projective objects, and hence $P(x) = \mathrm{ind} (\mathcal{O}_x)$ is projective in $\mathcal{O} \Mod$. For details, please refer to \cite[Corollary 2.6 and Lemma 2.7]{DLLY}.

Since $V$ is right perpendicular to all $J$-torsion $\mathcal{O}$-modules by the previous theorem and $P(x) / \underline{S}$ is a $J$-torsion module, the first term and the last term vanish. Thus we have the desired isomorphism.

(2)$\Rightarrow$(1). By the previous theorem, it suffices to show that
\[
\Hom_{\mathcal{O} \Mod} (U, V) = 0 = \Ext^1_{\mathcal{O} \Mod} (U, V)
\]
for any $J$-torsion module $U$. The long exact sequence in the previous paragraph and the given isomorphism tell us that these identities hold for $U = P(x)/\underline{S}$ with $x \in \Ob(\C)$ and $S \in J(x)$. We claim that
\[
\mathcal{P} = \{P(x)/\underline{S} \mid x \in \Ob(\C), \, S \in J(x) \}
\]
is a set of generators of $\mathcal{T}(J)$. To see this, take an arbitrary $T \in \mathcal{T}(J)$. For every object $x$ in $\C$ and every $v\in T_x$, there exists a
covering sieve $S \in J(x)$ such that $S \cdot v=0$. Consider the natural $\mathcal{O}$-module homomorphism
\[
P(x) \longrightarrow T, \qquad \boldsymbol{1}_x \longmapsto v .
\]
Since every element of $\underline{S}$ is of the form $f \cdot \boldsymbol{1}_x$ with $f \in S$, its image under this homomorphism is $f \cdot v = 0$. Hence the homomorphism vanishes on $\underline{S}$ and therefore factors through the quotient, giving an $\mathcal{O}$-module homomorphism
\[
P(x)/\underline{S} \longrightarrow T, \qquad \overline{\boldsymbol{1}}_x \longmapsto v .
\]
Taking the direct sum of these homomorphisms over all pairs $(x,v)$ yields an epimorphism from a coproduct of objects $P(x)/\underline{S}$ onto $T$. Hence
the collection of such modules is a set of generators for $\mathcal{T}(J)$.

Consequently, we may construct a short exact sequence
\[
0 \to U' \to Q \to U \to 0
\]
where $Q$ is a coproduct of $\mathcal{O}$-modules in $\mathcal{P}$. Consider the long exact sequence
\begin{align*}
& 0 \to \Hom_{\mathcal{O} \Mod} (U, V) \to \Hom_{\mathcal{O} \Mod} (Q, V) \to \Hom_{\mathcal{O} \Mod} (U', V)\\
& \to \Ext^1_{\mathcal{O} \Mod} (U, V) \to \Ext^1_{\mathcal{O} \Mod} (Q, V).
\end{align*}
Note that $Q$ is $J$-torsion, while $V$ is $J$-saturated by the previous theorem. Therefore, the second and the fifth terms vanish. It follows that the first term vanishes as well. But since $U$ can be any object in $\mathcal{T}(J)$ and $U'$ is contained in $\mathcal{T}(J)$, the third term also vanishes. Consequently, $\Ext^1_{\mathcal{O} \Mod} (U, V) = 0$, and the proof is complete.

(1)$\Leftrightarrow$(3). Note that $V$ is a sheaf of modules on $(\BC^{\op}, \mathcal{O})$ if and only if it is a sheaf of abelian groups; that is, viewed as a $\underline{\mathbb{Z}}$-module it is a sheaf of modules on $(\BC^{\op}, \underline{\mathbb{Z}})$. The equivalence between (1) and (3) then follows from this observation as well as the equivalence between (1) and (2).
\end{proof}

The second one gives a characterization of injective sheaves.

\begin{corollary} \label{injective sheaves}
Injective objects in $\Sh(\BC^{\op}, \mathcal{O})$ coincide with injective $J$-torsion free $\mathcal{O}$-modules.
\end{corollary}

\begin{proof}
This result has been proved in \cite[Corollary 3.9]{DLLX} for the atomic Grothendieck topology. Its proof still works in full generality.
\end{proof}

Let $\mathcal{O} \Mod^{\sat}$ be the full subcategory of $\mathcal{O} \Mod$ consisting of $J$-saturated $\mathcal{O}$-modules. The third consequence of Theorem \ref{characterize sheaves} provides the following equivalence:

\begin{corollary} \label{equivalence}
One has the following identification and equivalence of categories:
\[
\xymatrix{
\Sh(\BC^{\op}, \mathcal{O}) = \mathcal{O} \Mod^{\sat} \ar[r]^-{\simeq} & \mathcal{O} \Mod / \mathcal{T}(J).
}
\]
\end{corollary}

\begin{proof}
The identification follows from Theorem \ref{characterize sheaves} \footnote{We shall remind the reader that the abelian structure on $\mathcal{O} \Mod^{\sat}$ is not inherited from the abelian structure on $\mathcal{O} \Mod$. Indeed, given a monomorphism $\alpha: V \to W$ between two $J$-saturated modules, its cokernel in $\mathcal{O} \Mod$ might not coincide with the cokernel in $\mathcal{O} \Mod^{\sat}$.}, while the equivalence follows from \cite[Proposition 2.2(d)]{GLen}.
\end{proof}

Theorem \ref{characterize sheaves} also provides a new description of the sheafification functor $\sharp$, which is more elementary compared to the ones using local isomorphisms or dense monomorphisms. This description has been observed for the atomic Grothendieck topology in \cite[Section 3]{DLLX}, and here we extend it to all Grothendieck topologies. Explicitly, given an $\mathcal{O}$-module $V$, the first step is to take the $J$-torsion free part $V_F = V/\mathcal{T}_J(V)$ and embed it into an injective hull $I$ (which is $J$-saturated) to get a short exact sequence $0 \to V_F \to I \to C \to 0$ of $\mathcal{O}$-modules. In the second step, we construct the pull back of $I \to C$ and $C_T = \mathcal{T}_J(C) \to C$ to obtain the following commutative diagram such that all rows and columns are exact:
\[
\xymatrix{
0 \ar[r] & V_F \ar[r] \ar@{=}[d] & \widetilde{V} \ar[r] \ar[d] & C_T \ar[d] \ar[r] & 0\\
0 \ar[r] & V_F \ar[r] & I \ar[d] \ar[r] & C \ar[r] \ar[d] & 0\\
 & & C_F \ar@{=}[r] & C_F.
}
\]
Then $V^{\sharp} \cong \widetilde{V}$. For more details, please refer to \cite[Section 3]{DLLX}.

Now we consider sheaf cohomology groups. Given an object $x \in \Ob(\C)$, we can define a functor
\[
\Gamma_x^p: \mathcal{O} \Mod \to \mathcal{O}_x \Mod, \quad V \mapsto V_x,
\]
which is exact. Restricting it to the full subcategory $\Sh(\BC^{\op}, \mathcal{O})$, we get a functor $\Gamma_x$ which is only left exact since the abelian structure of $\Sh(\BC^{\op}, \mathcal{O})$ is not inherited from that of $\mathcal{O} \Mod$. Thus one can define $\mathrm{R}^i \Gamma_x(V)$, the $i$-th right derived functor of $\Gamma_x$, called the $i$-th \textit{sheaf cohomology group} of $V$ at $x$. For details, please refer to \cite[I.21.2]{Stack}.

Sheaf cohomology groups can be computed via injective resolutions. The following result relates them to right derived functors of $\mathcal{T}_J$.

\begin{corollary} \label{sheaf cohomology}
For $i \geqslant 1$, one has the following isomorphism of functors
\[
\mathrm{R}^i \Gamma_x \cong \Gamma_x^p \circ (\mathrm{R}^{i+1} \mathcal{T}_J).
\]
\end{corollary}

\begin{proof}
It has been proved in \cite[Theorem 3.16]{DLLX} for the special case that $J$ is the atomic Grothendieck topology and $\mathcal{O}$ is a constant structure sheaf. The proof actually works in the general situation.
\end{proof}

It is well known that morphisms between ringed sites induce geometric morphisms between categories of sheaves of modules. In the rest of this section we give an illustration of this fact via our torsion theoretic interpretation. Let $(\Co, \, J, \, \mathcal{O})$ and $(\D^{\op}, \, K, \, \mathcal{O}')$ be two ringed sites. To simplify the technical issue we suppose that both $\mathcal{O}$ and $\mathcal{O}'$ are the constant structure sheaf $\underline{\mathbb{Z}}$ induced by $\mathbb{Z}$. This simplification does not cause essential influence to the general case since a sheaf of modules on $(\Co, \, J, \, \mathcal{O})$ is an $\mathcal{O}$-module as well as a sheaf of abelian groups.

Recall that a functor $\phi: \Co \to \D^{\op}$ is \textit{cover-preserving} if for every $x \in \Ob(\C)$ and every $S \in J(x)$, the sieve generated by $\phi(S) = \{\phi(f) \mid f \in S \}$ is contained in $K(\phi(x))$.

\begin{proposition}
Suppose that $\Co$ and $\D^{\op}$ have finite limits and $\phi: \Co \to \D^{\op}$ is cover-preserving and preserves finite limits. Then $\phi$ induces a geometric morphism
\[
\phi^{\ast}: \Sh(\D^{\op}, \, K, \, \underline{\mathbb{Z}}) \to \Sh(\Co, \, J, \, \underline{\mathbb{Z}}).
\]
\end{proposition}

\begin{proof}
This is essentially a combination of \cite[C2.3, Lemma 2.3.3]{Jo} and \cite[C2.3, Corollary 2.3.4]{Jo}, so we only give a sketch of the proof. The functor $\phi$ induces a restriction functor
\[
\tilde{\phi}^{\ast}: \PSh(\D^{\op}, \, \underline{\mathbb{Z}}) \to \PSh(\Co, \, \underline{\mathbb{Z}})
\]
between the two presheaf categories as well as a left adjoint
\[
\phi_!: \PSh(\Co, \, \underline{\mathbb{Z}}) \to \PSh(\D^{\op}, \, \underline{\mathbb{Z}})
\]
given by the left Kan extension along $\phi$. Since $\phi$ preserves finite limits, $\phi_!$ is exact. Thus it suffices to show that $\tilde{\phi}^{\ast}$ restricts to a functor
\[
\phi^{\ast}: \Sh(\D^{\op}, \, K, \, \underline{\mathbb{Z}}) \to \Sh(\Co, \, J, \, \underline{\mathbb{Z}}).
\]
In this case, it has an exact left adjoint given by the composite of
\[
\xymatrix{
\Sh(\Co, \, J, \, \underline{\mathbb{Z}}) \ar[r]^-{\textsl{inc}} & \PSh(\Co, \, \underline{\mathbb{Z}}) \ar[r]^-{\phi_!} & \PSh(\D^{\op}, \, \underline{\mathbb{Z}}) \ar[r]^-{\sharp} & \Sh(\D^{\op}, \, K, \, \underline{\mathbb{Z}}).
}
\]

To complete the proof, we need to show that $\tilde{\phi^{\ast}}$ sends $K$-saturated modules to $J$-saturated modules. Since $(\phi_!, \, \tilde{\phi}^{\ast})$ is an adjoint pair and $\phi_!$ is exact, it suffices to show that $\phi_!$ sends $J$-torsion modules to $K$-torsion modules, which happens if and only if $\tilde{\phi}^{\ast}$ sends $K$-torsion free modules to $J$-torsion free modules. This is clear. Indeed, suppose that $V$ is $K$-torsion free. If there exist $x \in \Ob(\C)$ and $v \in (\tilde{\phi}^{\ast} V)_x$ such that $v$ is $J$-torsion, then one can find a covering sieve $S \in J(x)$ such that $S \cdot v = 0$. But this means that $\phi(f)$ sends $v \in V_{\phi(x)}$ to 0 for each $f \in S$. Consequently, $v$ is $K$-torsion since the covering sieve in $K(\phi(x))$ generated by $\phi(S)$ sends it to 0. This forces $v = 0$, so $\tilde{\phi}^{\ast}V$ is $J$-torsion free.
\end{proof}

\begin{remark}
In the above proposition the conditions that $\Co$ and $\D^{\op}$ have finite limits and $\phi$ preserves finite limits are used to show the exactness of the left adjoint $\phi_!$. Since $\phi_!$ is precisely the tensor functor $- \otimes_{\mathbb{Z} \D^{\op}} \mathbb{Z} \C^{\op}$, for categories without finite limits, one can use the following condition: the category algebra $\mathbb{Z} \C^{\op}$ is a flat right $\mathbb{Z} \D^{\op}$-module.
\end{remark}

\section{Grothendieck topologies on artinian EI categories}

In the previous section we have characterized sheaves of modules on $(\BC^{\op}, \mathcal{O})$ as $J$-saturated $\mathcal{O}$-modules and obtained the following equivalence
\[
\Sh(\BC^{\op}, \mathcal{O}) \simeq \mathcal{O} \Mod/\mathcal{T}(J).
\]
The Serre quotient on the right hand side is still complicated, so we want to find a more explicit description for some special Grothendieck topologies $J$. When $J$ is a \textit{subcategory topology} (or called \textit{rigid topology}, see \cite[Definition C.2.2.18]{Jo}), this Serre quotient is equivalent to $\mathcal{O}_{\mathscr{D}} \Mod$, where $\mathcal{O}_{\mathscr{D}}$ is the restriction of the structure sheaf $\mathcal{O}$ to a full subcategory $\mathscr{D}^{\op}$ of $\C^{\op}$,\footnote{To simplify the notation, we denote this restriction by $\mathcal{O}_{\D}$ rather than $\mathcal{O}_{\D^{\op}}$. This is reasonable because we can view $\mathcal{O}$ as a covariant functor from $\C$ to the category of rings, and hence the restriction is a covariant functor from $\D$ to the category of rings.} and hence becomes much more transparent. Furthermore, it has been pointed out in \cite[IV.9]{AGV} that all Grothendieck topologies on a finite Karoubi category are rigid, and this result also holds for artinian posets (see \cite[Theorem 2.12]{Lin}).

The main goal of this section is to prove the above conclusion for $\C^{\op}$ when $\C$ is an EI category whose underlying poset is noetherian; or equivalently, the underlying poset of $\C^{\op}$ is artinian. In particular, we generalize the corresponding results for finite acyclic quivers by Murfet \cite{Mur} and for posets by Hemelaer \cite{Hem} and Lindenhovius \cite{Lin}.

\subsection{Definitions and the main result}

Given a skeletal small category $\C$, we impose the following relation $\leqslant$ on $\Ob(\C)$: $x \leqslant y$ if $\C(x, y) \neq \varnothing$. This relation is reflexive and transitive, but in general is not antisymmetric.

\begin{definition}
A small skeletal category $\C$ is called \textit{directed} if the relation $\leqslant $ on $\Ob(\C)$ defined above is a partial order. A directed category $\C$ is called \textit{noetherian} if the poset $(\Ob(\C), \leqslant)$ is noetherian; that is, $\C$ does not contain an infinite sequence of morphisms
\[
\xymatrix{
x_0 \ar[r] & x_1 \ar[r] & x_2 \ar[r] & \ldots
}
\]
such that all objects appearing in it are distinct. Dually, one defines \textit{artinian} directed categories.
\end{definition}

\begin{remark} \label{filtration}
When $\C$ is a directed category, one can construct a natural filtration on $\Ob(\C)$ as follows. Explicitly, let $O_0 = \varnothing$, $O_1$ be the set of objects in $\Ob(\C)$ minimal with respect to $\leqslant$. For an ordinal $\chi$, let $O_{\chi}$ be the union of $O_{\chi - 1}$ and the set of minimal elements in $\Ob(\C) \backslash O_{\chi - 1}$ provided that $\chi$ is a successor ordinal, and let $O_{\chi}$ be the union of $O_{\lambda}$ for $\lambda < \chi$ if $\chi$ is a limit ordinal. In this way we obtain a chain $O_1 \subseteq O_2 \subseteq \cdots \subseteq \Ob(\C)$. Note that $\C$ is artinian if and only if there exists an ordinal $\zeta$ such that $O_{\zeta} = \Ob(\C)$. Dually, by taking maximal objects one gets another filtration $\{O'_{\chi} \}$ on $\Ob(\C)$, and $\C$ is noetherian if and only if there exists an ordinal $\zeta$ such that $O'_{\zeta} = \Ob(\C)$. For details, see \cite[Subsection 3.1]{DLLY}.
\end{remark}

\begin{definition}
Let $\C$ be a directed category. A full subcategory $\mathscr{D}$ is called a \textit{coideal} of $\C$ if it is downward closed: given $x \in \Ob(\mathscr{D})$ and $y \in \Ob(\C)$, if $y \leqslant x$, then $y \in \Ob(\mathscr{D})$ as well. Dually, one can define \textit{ideals} of $\C$.
\end{definition}

The above concepts generalize downward sets and upward sets of posets.

\begin{definition}
Let $\C$ be a directed category. We say that $x \in \Ob(\C)$ is \textit{downward unbounded} if there exists a descending chain
\[
x > x_1 > x_2 > \ldots
\]
of infinite length. Otherwise, $x$ is called \textit{downward bounded}. \textit{Upward unbounded} objects and \textit{upward bounded} objects are defined similarly.
\end{definition}

Let $\C$ be a directed category. Denote by $\C_{db}$ (resp., $\C_{du}$) the full subcategory of $\C$ consisting of downward bounded (resp., downward unbounded) objects. Similarly, one defines $\C_{ub}$ and $\C_{uu}$. The following result can be easily checked, so we omit its proof. The reader can also formulate a dual result for the upward bounded/unbounded case.

\begin{lemma} \label{combinatorics}
Let $\C$ be a directed category and $x$ an object in $\C$. One has:
\begin{enumerate}
\item $\C$ is artinian if and only if $\C = \C_{db}$;

\item $x$ is downward bounded if and only if $\C(y, x) = \varnothing$ for every downward unbounded object $y$;

\item $x$ is downward unbounded if and only if $\C(x, y) = \varnothing$ for every downward bounded object $y$;

\item $\C_{du}$ (resp., $\C_{db}$) is an ideal (resp., a coideal) of $\C$.
\end{enumerate}
\end{lemma}

From this lemma we deduce an immediate result: if $x$ is downward bounded (resp., downward unbounded), then $\C_{db}(-, x) = \C(-, x)$ (resp., $\C_{du} (x, -) = \C(x, -)$) as sets of morphisms.

Now we introduce the rigid topology and the main theorem of this section. Let $\C$ be a small skeletal category, and $J$ a Grothendieck topology on $\Co$. Recall from \cite[C2.2]{Jo} that $x \in \Ob(\C)$ is called \textit{$J$-irreducible} if $J(x) = \{ \C(x, -) \}$. For $y \in \Ob(\C)$, let $S_y$ be the sieve generated by the family of morphisms (in $\C$ rather than $\Co$) from $y$ to all $J$-irreducible objects. It can be explicitly described as follows
\[
S_y = \bigcup_{x \text{ irreducible}} \C(x, -) \circ \C(y, x)
\]
where $\C(x, -) \circ \C(y, x) = \{ f \circ g \mid f: x \to \bullet \text{ and } g: y \to x\}$.

\begin{definition} \label{def of rigid topologies}
A Grothendieck topology $J$ on $\Co$ is called \textit{rigid} if $S_y \in J(y)$ for every $y \in \Ob(\C)$.
\end{definition}

A rigid Grothendieck topology $J$ is also called a \textit{subcategory topology} with respect to the full subcategory consisting of $J$-irreducible objects. When $J$ is a rigid topology, $S_y$ is the unique minimal covering sieve in $J(y)$, so $J(y) = \{ S \subseteq \C(y, -) \mid S \supseteq S_y \}$. For details, see \cite[C2.2, Example 2.2.4(d), Remark 2.2.5, and Definition 2.2.18]{Jo}.

Recall that $\C$ is an \emph{EI-category} if every endomorphism in $\C$ is an isomorphism, or equivalently, the endomorphism monoid of every object of $\C$ is a group. The main theorem of this section is:

\begin{theorem} \label{rigid topologies}
Let $\C$ be a directed category. Then every Grothendieck topology on $\Co$ is rigid if and only if $\C$ is a noetherian EI category.
\end{theorem}

\begin{remark}
By \cite[C2.2, Lemma 2.2.21]{Jo} and the remark after its proof, if $\C$ is a Cauchy-complete category containing only finitely many morphisms, or more generally, if the coslice category $x \backslash \C$ is equivalent to a finite category for $x \in \Ob(\C)$, then every Grothendieck topology on $\Co$ is rigid. This does not trivially imply the sufficient direction of the previous theorem. For instance, let $\C$ be a poset such that $\Ob(\C) = \{ x_i \mid i \in \mathbb{N} \} \sqcup \{ - \infty \}$, $x_i > - \infty$ for each $i$, and $x_i$ and $x_j$ are not comparable for $ i \neq j$. Then $\C$ is a noetherian EI category, but the coslice category $-\infty \backslash \C$ is not equivalent to a finite category.
\end{remark}

The proof of the above theorem is complicated and lengthy, so we give it in the next two subsections.

\subsection{The necessary direction}

In this subsection we prove the only if direction of Theorem \ref{rigid topologies}. Throughout this subsection let $\C$ be a directed category. As the starting point, we consider the special case that $\C$ has only one object, so we can identify $\Co$ with a monoid $M$, and correspondingly, a sieve on $\Co$ is nothing but a right ideal of $M$.

\begin{lemma} \label{topologies on monoids}
A monoid $M$ has exactly two Grothendieck topologies if and only if $M$ is a group.
\end{lemma}

\begin{proof}
If $M$ is a group, then the only right ideals of $M$ are $M$ itself and the empty set. Consequently, there are exactly two Grothendieck topologies on $M$, namely the trivial topology and the maximal topology.

Conversely, suppose that $M$ is not a group, and let
\[
N=\{m\in M\mid m\text{ is not right invertible}\}.
\]
We first observe that $N$ is nonempty. Indeed, if every element of $M$ were right invertible, then $M$ would be a group. To see this, let $a \in M$ and
choose $b, c \in M$ such that $ab=1$ and $bc=1$. Then
\[
a=a(bc)=(ab)c=c,
\]
and hence
\[
ba=bc=1.
\]
Thus every element of $M$ would be invertible, contradicting the assumption that $M$ is not a group.

It is straightforward to check that $N$ is a proper right ideal of $M$. Furthermore, $N$ is contained in the dense topology $J_d$ on $M$, which by definition consists of all right ideals $S$ of $M$ satisfying the condition that for every $m\in M$, there exists $n\in M$ such that $mn\in S$. Indeed, given $m\in M$, if $m$ is not right invertible, then $m\cdot 1\in N$. If $m$ is right invertible, choose a right inverse $n\in M$ of $m$. Since $N$ is nonempty, choose $l\in N$. Then
\[
m(nl)=l\in N.
\]
Therefore $N\in J_d$.

Note that $J_d$ is different from the maximal topology, since it does not contain the empty set. It is also different from the trivial topology, since it contains the proper subset $N$ of $M$. Therefore, there are at least three different Grothendieck topologies on $M$.
\end{proof}

Now we describe two technical lemmas playing a key role for proving the necessary direction of Theorem \ref{rigid topologies}.

\begin{lemma} \label{topologies on coideals}
Let $J$ be a Grothendieck topology on $\Co$ and $\D$ an ideal of $\C$. Then the restriction of $J$ on $\D^{\op}$ given by
\[
J_{\D} (x) = \{S \cap \D(x, -) \mid S \in J(x) \}
\]
for $x \in \Ob(\D)$ defines a Grothendieck topology on $\D^{\op}$. Moreover, every Grothendieck topology on $\D^{\op}$ can be obtained in this way. Consequently, if every Grothendieck topology $J$ on $\Co$ is rigid, so is every Grothendieck topology on $\D^{\op}$.
\end{lemma}

\begin{proof}
Since $\D$ is an ideal of $\C$, we have $\C(x, y) = \varnothing$ for $x \in \Ob(\D)$ and $y \notin \Ob(\D)$. Consequently, as sets of morphisms, we have $S \cap \D(x, -) = S$ for $S \in J(x)$ and $x \in \Ob(\D)$. With this observation, it is easy to check that the restriction $J_{\D}$ is indeed a Grothendieck topology on $\D^{\op}$.

Conversely, given a Grothendieck topology $K$ on $\D^{\op}$, we define a rule $J$ for $\Co$ by setting $J(x) = K(x)$ for $x \in \Ob(\D)$ and $J(y) = \{ \C(y, -) \}$ for $y \in \Ob(\C) \setminus \Ob(\D)$. It is routine to check that $J$ is a Grothendieck topology on $\Co$, and furthermore, its restriction on $\D^{\op}$ is precisely $K$.

To establish the last statement, we need to show that the restriction $J_{\D}$ is a rigid Grothendieck topology on $\D^{\op}$ whenever $J$ is a rigid Grothendieck topology on $\Co$. But this is clear. Indeed, since $J$ is rigid, for every $x \in \Ob(\D)$, morphisms in $\C$ from $x$ to $J$-irreducible objects generate a covering sieve $S_x \in J(x)$. Since $\D$ is an ideal of $\C$, as explained before, $J_{\D}$-irreducible objects are precisely $J$-irreducible objects contained in $\Ob(\D)$, and only these objects contribute to $S_x$. Consequently, $S_x = S_x \cap \D(x, -) \in J_{\D}(x)$ is generated by morphisms in $\D$ from $x$ to $J_{\D}$-irreducible objects, so $J_{\D}$ is rigid.
\end{proof}

The second lemma gives a dual result: if every Grothendieck topology on $\Co$ is rigid, then so is every Grothendieck topology on $\mathscr{E}^{\op}$ where $\mathscr{E}$ is a coideal of $\C$. To prove this result, we need to introduce a slightly more complicated construction.

Let $\mathscr{E}$ be a coideal of $\C$, and let $\D$ be the full subcategory of $\C$ consisting of those objects not contained in $\Ob(\mathscr{E})$. It is easy to check that $\D$ is an ideal of $\C$. Let $K$ be the maximal topology on $\D^{\op}$ and $L$ an arbitrary Grothendieck topology on $\mathscr{E}^{\op}$. We define a rule $J$ on $\Co$ as follows:
\[
J(x) = \begin{cases}
K(x), & x \in \Ob(\D);\\
\{ S \subseteq \C(x, -) \mid S \cap \mathscr{E}(x, -) \in L(x) \}, & x \in \Ob(\mathscr{E}).
\end{cases}
\]
We point out that $S \subseteq \C(x, -)$ is contained in $J(x)$ if and only if there is a certain $S' \in L(x)$ such that $S' \subseteq S$. The only if direction is trivial by taking $S' = S \cap \mathscr{E}(x, -)$. The if direction is also clear: $S' \subseteq S$ implies that $S \cap \mathscr{E}(x, -) \supseteq S'$, so $S \cap \mathscr{E}(x, -)$ is contained in $L(x)$.

\begin{lemma} \label{topologies on ideals}
The rule $J$ defined above is a Grothendieck topology on $\Co$. Furthermore, if all Grothendieck topologies on $\Co$ are rigid, then so is every Grothendieck topology on $\mathscr{E}^{\op}$.
\end{lemma}

\begin{proof}
The maximal axiom clearly holds. For the stability axiom, let $S \in J(x)$ be a covering sieve, and $f: x \to y$ be a morphism in $\C$. If $y$ is an object in $\D$, then $f^{\ast} (S) \subseteq \C(y, -)$ is contained in $K(y) = J(y)$ since $K$ is the maximal topology on $\D^{\op}$ and $\D(y, -) = \C(y, -)$ as $\D$ is an ideal of $\C$. If $y$ is contained in $\mathscr{E}$, so is $x$ because $\mathscr{E}$ is a coideal. Since $S \cap \mathscr{E}(x, -) \in L(x)$ by definition of $J$, in this case
\[
f^{\ast}(S) \cap \mathscr{E}(y, -) = f^{\ast}(S) \cap f^{\ast}(\mathscr{E}(x, -)) = f^{\ast} (S \cap \mathscr{E}(x, -))
\]
is contained in $L(y)$ by the stability axiom on $L$, and hence $f^{\ast}(S) \in J(y)$ as desired.

To check the transitivity condition, take $T \subseteq \C(x, -)$, $S \in J(x)$ such that $f^{\ast}(T) \in J(y)$ for each $f: x \to y$ in $S$. If $x$ is contained in $\mathscr{D}$, then clearly $T \in J(x)$. Otherwise, one has
\[
T \supseteq \bigcup_{f \in S} f^{\ast}(T) \circ f \supseteq \bigcup_{f \in S \cap \mathscr{E}(x, -)} f^{\ast}(T) \circ f \supseteq \bigcup_{f: x \to y \atop f \in S \cap \mathscr{E}(x, -)} (f^{\ast}(T) \cap \mathscr{E}(y, -)) \circ f.
\]
Denote the right side of the above formula by $T'$, which is a subfunctor of $\mathscr{E}(x, -)$ since $x \in \Ob(\mathscr{E})$. Note that $S \cap \mathscr{E}(x, -) \in L(x)$ and $f^{\ast}(T) \cap \mathscr{E}(y, -) \in L(y)$ by definition of $J$, and $f^{\ast}(T) \cap \mathscr{E}(y, -)$ is a subset of $f^{\ast} (T')$. It follows that $f^{\ast} (T') \in L(y)$ for each $f \in S \cap \mathscr{E}(x, -)$. Applying the transitivity axiom for $L$ to $S \cap \mathscr{E}(x, -) \in L(x)$ and $T' \subseteq \mathscr{E}(x, -)$, we deduce that $T' \in L(x)$. Consequently, $T \in J(x)$ since it contains the $L$-covering sieve $T'$.

Now we prove the second statement. Let $L$ be a Grothendieck topology on $\mathscr{E}^{\op}$ and construct a Grothendieck topology $J$ on $\Co$ as above. Note that for every $x \in \Ob(\mathscr{E})$ one has
\[
L(x) = \{ S \cap \mathscr{E}(x, -) \mid S \in J(x) \}.
\]
Indeed, given a covering sieve $T \in L(x)$ in $\mathscr{E}^{\op}$, it generates a sieve $S \subseteq \C(x, -)$ in $\Co$, which clearly belongs to $J(x)$ since $S \supseteq T$. It is easy to check that $S \cap \mathscr{E}(x, -) = T$, so the left side of the above identity is a subset of the right side. The inclusion of the other direction follows from the construction of $J$.

Since $J$ is rigid, for every $x \in \Ob(\C)$, morphisms in $\C$ from $x$ to $J$-irreducible objects form a covering sieve $S_x \in J(x)$. Note that by our construction of $J$, every object contained in $\D$ is not $J$-irreducible since $K$ is the maximal topology on $\D$. Therefore, $S_x$ is generated by morphisms in $\C$ from $x$ to $J$-irreducible objects contained in $\mathscr{E}$, which coincide with $L$-irreducible objects. Consequently, for $x \in \Ob(\mathscr{E})$, the covering sieve $(S_x \cap \mathscr{E}(x, -)) \in L(x)$ is generated by morphisms in $\mathscr{E}$ from $x$ to $L$-irreducible objects. Therefore, $L$ is rigid.
\end{proof}

\begin{remark}
The proof actually shows a slightly stronger statement. Namely, if the Grothendieck topology $J$ on $\C^{\op}$ induced by a Grothendieck topology $L$ on $\mathscr{E}^{\op}$ is rigid, then $L$ is rigid as well. The assertion of the lemma follows immediately by the assumption that every Grothendieck topology on $\mathscr{C}^{\op}$ is rigid.
\end{remark}

We are ready to prove the necessary direction of Theorem \ref{rigid topologies}.

\begin{proposition} \label{necessary direction}
Let $\C$ be a directed category. If each Grothendieck topology $J$ on $\Co$ is rigid, then $\C$ is a noetherian EI category.
\end{proposition}

\begin{proof}
We first show that $\C$ is an EI category; that is, for every $x \in \Ob(\C)$, the monoid $M_x = \C(x, x)$ is a group. Let $\D$ be the full subcategory of $\C$ consisting of objects $y$ with $y \geqslant x$. Clearly, $\D$ is an ideal of $\C$. By Lemma \ref{topologies on coideals}, every Grothendieck topology on $\D^{\op}$ shall be rigid. Now let $\mathscr{E}$ be the full subcategory of $\D$ consisting of the single object $x$. Then $\mathscr{E}$ is a coideal of $\D$. By Lemma \ref{topologies on ideals}, every Grothendieck topology on $\mathscr{E}^{\op}$ shall be rigid as well. By Lemma \ref{topologies on monoids}, the monoid $M_x^{\op}$ is a group, so is $M_x$. Therefore, $\C$ is an EI category.

Next we show that $\C$ is noetherian by contradiction. If this is not true, then the full subcategory $\C_{uu}$ consisting of upward unbounded objects is nonempty, and is a coideal of $\C$ by the dual version of Lemma \ref{combinatorics}. By Lemma \ref{topologies on ideals}, every Grothendieck topology on $\C_{uu}^{\op}$ shall be rigid. We claim that the dense topology $J_d$ on $\C_{uu}^{\op}$ is not rigid, contradicting our assumption that all topologies are rigid.

Given $x \in \Ob(\C_{uu})$, define
\[
T_x = \bigsqcup_{y \in \Ob(\C_{uu}) \atop y \neq x} \C(x, y),
\]
a proper subfunctor of $\C_{uu}(x, -)$. Since $x$ is upward unbounded, there are infinitely many objects $y \in \Ob(\C_{uu})$ such that $\C_{uu}(x, y)$ is nonempty, so $T_x$ is nonempty. Moreover, it is easy to see that $T_x \in J_d(x)$. Therefore, $x$ is not $J_d$-irreducible, so no object in $\C_{uu}$ is $J_d$-irreducible. Consequently, $J_d$ is not rigid because the subcategory topology $J$ on $\C_{uu}$ such that no object is $J$-irreducible is the maximal topology, which is different from $J_d$.
\end{proof}

\subsection{Sufficient direction}

In this subsection we prove the sufficient direction of Theorem \ref{rigid topologies}. Throughout this subsection let $\C$ be a noetherian EI category unless otherwise specified. We first show that for every Grothendieck topology $J$ on $\Co$ and every $x \in \Ob(\C)$, $J(x)$ has a unique minimal element: the intersection of all $S \in J(x)$.

\begin{lemma} \label{intersection}
Every Grothendieck topology $J$ on $\Co$ is closed under arbitrary intersections.
\end{lemma}

\begin{proof}
Since $J$ is closed under upward inclusions, it suffices to show
\[
\bigcap_{S \in J(x)} S \in J(x)
\]
for all $x \in \Ob(\C)$. If the conclusion is false, then the set $O$ consisting of objects $x$ such that the above formula fails is not empty. Since $\Ob(\C)$ is a noetherian poset, we can find a maximal object $x$ in $O$. We deduce that:
\begin{itemize}
\item $\varnothing \notin J(x)$, since otherwise the formula holds;

\item $x$ is not a maximal object in $\C$, since otherwise $\C(x, -)$ has only two subfunctors, namely itself and the empty functor, and the formula also holds;

\item the cardinality $|J(x)|$ must be infinite since Grothendieck topologies are closed under finite intersections.
\end{itemize}

Take an arbitrary $T$ in $J(x)$. We may assume that $T$ is different from $\C(x, -)$ (this is possible since $J(x)$ is an infinite set). By the transitivity axiom, there must exist a certain $f: x \to y$ in $T$ such that
\[
f^{\ast} (\bigcap_{S \in J(x)} S) = \bigcap_{S \in J(x)} f^{\ast}(S)
\]
is not contained in $J(y)$. But clearly we have
\[
\left(\bigcap_{S\in J(x)} f^{\ast}(S)\right) \supseteq \bigcap_{T\in J(y)} T .
\]
Consequently,
\[
\left( \bigcap_{T\in J(y)} T \right) \notin J(y).
\]
that is, $y \in O$ as well. However, because $\C$ is a skeletal EI category, $f$ cannot be an endomorphism since otherwise $f$ is an automorphism and hence $T = \C(x, -)$. So, we have $x < y$, contradicting the assumption that $x$ is maximal in $O$. The conclusion then follows by contradiction.
\end{proof}

By this lemma, a Grothendieck topology $J$ on $\Co$ determines a family of minimal covering sieves each of which belongs to $J(x)$ for a distinct $x \in \Ob(\C)$. Now we consider the converse procedure, generalizing the idea in \cite{Mur} to construct Grothendieck topologies on $\C^{\op}$ by giving a family of minimal covering sieves satisfying a special requirement. We assign to each object $x$ in $\C$ a subfunctor $S_x \subseteq \C(x, -)$, and call this assignment \textit{consistent} if the following condition holds: for every $x \in \Ob(\C)$, either $S_x = \C(x, -)$ or
\begin{equation} \label{consistent}
S_x = \bigcup_{x \neq y \in \Ob(\C) \atop f \in \C(x, y)} S_yf = \bigcup_{x \neq y \in \Ob(\C)} S_y \circ \C(x, y)
\end{equation}
where $S_yf = \{ gf \mid g \in S_y \}$. By convention, the above union is empty when the index is empty. Such a consistent family $\{ S_x \}_{x \in \Ob(\C)}$ determines a rule $J$ by
\[
J(x) = \{S \subseteq \C(x, -) \mid S \supseteq S_x \}.
\]

\begin{proposition} \label{classify tops}
The rule $J$ defined above is a Grothendieck topology on $\Co$, and this construction exhausts all Grothendieck topologies on $\Co$.
\end{proposition}

The statement that the above construction exhausts all Grothendieck topologies means that every Grothendieck topology $J$ on $\mathcal{C}^{\op}$ is obtained from a unique consistent family. More precisely, the minimal covering sieve
\[
S_x = \bigcap_{S \in J(x)} S
\]
recovers the consistent family used in the construction above, and the original topology is recovered by
\[
J(x) = \{S \subseteq \mathcal{C}(x,-) \mid S \supseteq S_x\}.
\]
Thus Grothendieck topologies on $\mathcal{C}^{\op}$ are in bijection with consistent families of sieves.

\begin{proof}
We show that $J$ is indeed a Grothendieck topology on $\C^{\op}$. The maximal axiom holds obviously. For the stability axiom, we observe that if $S \in J(x)$, then $S \supseteq S_x$, and for each morphism $f: x \to y$ in $\C$, one has
\[
f^{\ast}(S) \supseteq f^{\ast} (S_x) \supseteq f^{\ast} (S_yf) \supseteq S_y,
\]
so $f^{\ast}(S) \in J(y)$. It remains to check the transitivity axiom.

Take a subfunctor $S \subseteq \C(x, -)$ and $T \in J(x)$ such that $f^{\ast}(S) \in J(y)$ for each morphism $f: x \to y$ in $T$. We want to check that $S \in J(x)$; that is, $S_x \subseteq S$. If $S = \C(x, -)$, the conclusion holds trivially. Similarly, if $T$ contains an endomorphism on $x$, then $T = \C(x, -)$, so by taking $f$ to be the identity morphism on $x$ we have $S = f^{\ast} (S) \in J(x)$. Now suppose that both $S$ and $T$ do not contain any endomorphism on $x$. In this case it suffices to show the following claim:
\begin{center}
$(\S)$ each $\alpha \in S_x$ can be written as $\alpha = \beta \gamma$ with $\beta \in S_y$ and $\gamma \in T$.
\end{center}
\noindent Indeed, since for each $\gamma \in T$ one has $\gamma^{\ast} (S) \in J(y)$, namely $\gamma^{\ast} (S) \supseteq S_y$, we deduce that $S \supseteq S_yT$. Thus if the above claim is valid, then $\alpha \in S_yT \subseteq S$ for each $\alpha \in S_x$, so $S_x \subseteq S$ as desired.

Now we check the claim $(\S)$. Since $S \neq \C(x, -)$, by the consistency condition we can write $\alpha$ as a composite $\beta_1 \gamma_1$ with $\gamma_1 \in \C(x, x_1)$ for a certain object $x_1 \neq x$ and $\beta_1 \in S_{x_1}$. We have two cases:
\begin{itemize}
\item If $S_{x_1} = \C(x_1, -)$, then $\gamma_1 \in S_x \subseteq T$ by the consistency condition, so the claim holds.

\item Otherwise, equation (\ref{consistent}) holds for $S_{x_1}$. Thus one can replace $x$ by $x_1$ and write $\beta_1 = \beta_2 \gamma_2$ where $\gamma_2 \in \C(x_1, x_2)$ and $\beta_2 \in S_{x_2}$. Clearly, if $S_{x_2} = \C(x_2, -)$, then applying the argument as in the previous case, we deduce the claim.
\end{itemize}

Continuing the above procedure we obtain a sequence of objects $x < x_1 < x_2 < \ldots$ which must be finite since $\C$ is noetherian. Let $x_n$ be the last object, then either $S_{x_n} = \C(x_n, -)$ (and hence $(\S)$ holds) or $x_n$ is maximal with respect to $\leqslant$. For the second case, we have $\beta_n \in S_{x_n}$ is an automorphism since $\C(x_n, x_n)$ is assumed to be a group, so $S_{x_n} = \C(x_n, -)$ as well. Consequently, the claim $(\S)$ holds in both cases.

Now we check the second statement. Let $J$ be an arbitrary Grothendieck topology on $\Co$, and for every $x \in \Ob(\C)$ let $S_x$ be the intersection of all members in $J(x)$, which is contained in $J(x)$ by Lemma \ref{intersection}. It remains to show that the family $\{S_x \mid x \in \Ob(\C) \}$ is consistent; that is, equation (\ref{consistent}) holds provided that $S_x \neq \C(x, -)$. When $x$ is a maximal object, $\C(x, -)$ has only two subfunctors, namely itself and the empty functor. In this case it is easy to check the consistent condition.

Suppose that $x$ is not a maximal object. For each morphism $f: x \to y$ with $x \neq y$, one has $f^{\ast} (S_x) \in J(y)$, so $f^{\ast}(S_x) \supseteq S_y$, namely $S_x \supseteq S_y f$. Consequently, we obtain
\[
S_x \supseteq \bigcup_{x \neq y \in \Ob(\C) \atop f \in \C(x, y)} S_y f.
\]
To check the inclusion of the other direction, it suffices to show that the right side, denoted by $T$, is contained in $J(x)$. We prove it by the transitivity axiom. Indeed, for an arbitrary $g: x \to z$ in $S_x$, we have $z \neq x$ since $S_x \neq \C(x, -)$. Thus $S_z g \subseteq T$ and hence
\[
g^{\ast}(T) \supseteq g^{\ast} (S_z g) \supseteq S_z \in J(z).
\]
Since $J$ is closed under upward inclusions, we deduce that $g^{\ast} (T) \in J(z)$ as well. The transitivity axiom then tells us that $T \in J(x)$, as desired.
\end{proof}

The reader may guess that the Grothendieck topology $J$ on $\Co$ constructed before Proposition \ref{classify tops} is rigid, or equivalently, each $S_x$ in the consistent family is generated by morphisms in $\C$ from $x$ to $J$-irreducible objects. This is indeed the case by the following corollary, which establishes the sufficient direction of Theorem \ref{rigid topologies}.

\begin{corollary} \label{sufficient direction}
Every Grothendieck topology on $\Co$ is rigid.
\end{corollary}

\begin{proof}
Proposition \ref{classify tops} tells us that every Grothendieck topology on $\Co$ is determined by a consistent family. By comparing formula (\ref{consistent}) and Definition \ref{def of rigid topologies}, it suffices to show that each $S_x$ in formula (\ref{consistent}) coincides with the subfunctor $S'_x \subseteq \C(x, -)$ generated by all morphisms from $x$ to $J$-irreducible objects. When $S_x = \C(x, -)$, $x$ is $J$-irreducible, so this is trivially true.

Suppose that $x$ is not $J$-irreducible. In this case we have
\[
S_x = \bigcup_{x \neq y \in \Ob(\C) \atop f \in \C(x, y)} S_yf = \bigcup_{x \neq y \in \Ob(\C)} S_y \circ \C(x, y)
\]
and
\[
S_x' = \bigcup_{y \text{ irred.}} \C(y, -) \circ \C(x, y).
\]
Note that $S_y = \C(y, -)$ when $y$ is $J$-irreducible, so $S_x' \subseteq S_x$. We prove the other inclusion $S_x \subseteq S_x'$ via a transfinite induction based on a filtration $\{O'_{\chi}\}$ of maximal objects specified in Remark \ref{filtration}. Explicitly, there is an increasing sequence $O_1' \subseteq O_2' \subseteq O_3' \subseteq \ldots$ of subsets of $\Ob(\C)$ satisfying the following conditions:
\begin{itemize}
\item there exists an ordinal $\zeta$ such that $\Ob(\C) = O_{\zeta}'$ since $\C$ is noetherian;

\item $\C(z, w) = \varnothing$ if $w \in O'_{\lambda +1} \setminus O_{\lambda}'$ and $z \in O'_{\lambda}$ for a certain ordinal $\lambda$;

\item for $x \in \Ob(\C)$, there is a minimal ordinal $d(x)$ such that $x \in O'_{d(x)}$.
\end{itemize}

If $x$ is maximal with respect to the partial order $\leqslant$, i.e., $d(x) = 1$, then
\[
S_x = \bigcup_{x \neq y \in \Ob(\C) \atop f \in \C(x, y)} S_yf = \varnothing
\]
since there is no morphism $f: x \to y$ with $x \neq y$. Clearly, $S_x \subseteq S_x'$.

Suppose that $d(x) = \lambda$. Since $x$ is not $J$-irreducible, one has
\[
S_x = (\bigcup_{x \neq y \atop y \text{ irred.}} S_y \circ \C(x, y)) \cup (\bigcup_{x \neq y \atop y \text{ not irred.}} S_y \circ \C(x, y)) \subseteq S_x' \cup (\bigcup_{x \neq y \atop y \text{ not irred.}} S_y \circ \C(x, y)).
\]
But $x \neq y$ and $\C(x, y) \neq \varnothing$ imply $d(y) < d(x)$. By the induction hypothesis, $S_y \subseteq S_y'$, and hence
\[
S_y \circ \C(x, y) \subseteq S_y' \circ \C(x, y) = \bigcup_{z \text{ irred.}} \C(z, -) \circ \C(y, z) \circ \C(x, y) \subseteq \bigcup_{z \text{ irred.}} \C(z, -) \circ \C(x, z) \subseteq S_x'
\]
when $y$ is not $J$-irreducible. Consequently, one has $S_x \subseteq S_x'$. The conclusion then follows from a transfinite induction.
\end{proof}

When $\C$ is a noetherian EI category, the following result gives a very transparent description for $J$-torsion $\mathcal{O}$-modules and categories of sheaves of modules.

\begin{theorem} \label{classify sheaves}
Let $\C$ be a noetherian EI category, $J$ a Grothendieck topology on $\Co$, $\mathcal{O}$ a structure sheaf on $\BC^{\op}$, and $\mathscr{D}$ the full subcategory of $\C$ consisting of $J$-irreducible objects. Then:
\begin{enumerate}
\item an $\mathcal{O}$-module $V$ lies in $\mathcal{T} (J)$ if and only if $V_x = 0$ for each $x \in \Ob(\D)$;

\item consequently, one has
\[
\Sh(\BC^{\op}, \, \mathcal{O}) \simeq \mathcal{O}_{\D} \Mod,
\]
where $\mathcal{O}_{\D}$ is the presheaf of rings over $\D^{\op}$ obtained via restricting $\mathcal{O}$ to $\D^{\op}$.
\end{enumerate}
\end{theorem}

\begin{proof}
(1): The only if direction is clear. For the other direction, take an $\mathcal{O}$-module $V$ such that $V_x = 0$ for every $x \in \Ob(\D)$, an arbitrary $x \in \Ob(\C)$ and an element $v \in V_x$, we want to show that $S_x \cdot v = 0$. If $S_x = \C(x, -)$, then $V_x = 0$ and hence $v = 0$. If $S_x = \varnothing$, then this trivially holds by our convention, so we may assume that $S_x \neq \varnothing$. By the consistent condition,
\[
S_x = \bigcup_{x \neq y \in \Ob(\C) \atop f \in \C(x, y)} S_y f.
\]

Take a morphism $(g: x \to z) \in S_x$ and write it as $g = g_1f$ with $f: x \to x_1$ and $g_1: x_1 \to z$ in $S_{x_1}$. Then $g \cdot v = g_1 \cdot (f \cdot v)$. Replace $v$ by $f \cdot v$, $x$ by $x_1$, $g$ by $g_1$ and repeat the above argument. This procedure ends with two cases:
\begin{itemize}
\item We obtain a certain $x_n \neq z$ such that $S_{x_n} = \C(x_n, -)$. In this case the conclusion holds since $V_{x_n} = 0$.

\item After finitely many steps we obtain $x_n = z$ since $\C$ is noetherian. In this case, $g_n$ is a morphism from $z$ to itself, and hence an automorphism. Since $S_z$ contains $g_n$, it coincides with $\C(z, -)$, and the conclusion also holds since $V_z = 0$.
\end{itemize}
This finishes the proof of the first statement.

(2): This is essentially \cite[C2.2, Theorem 2.2.3]{Jo}. For the convenience of the reader, we give a sketched proof. Given an $\mathcal{O}$-module $V$, its restriction to $\D$ is clearly an $\mathcal{O}_{\D}$-module. In this way we obtain a restriction functor $\res: \mathcal{O} \Mod \to \mathcal{O}_{\mathscr{D}} \Mod$. It has a right adjoint given by the right Kan extension described as follows: given an $\mathcal{O}_{\D}$-module $W$ and an object $x \in \Ob(\C)$,
\[
(\mathrm{Ran}_{\iota} W)_x = \varprojlim_{\iota (y) \leftarrow x} W_y,
\]
where $\iota: \D \to \C$ is the inclusion functor, and the limit is taken over the comma category $(\iota \downarrow \mathrm{const}_x)$. When $x \in \Ob(\D)$, it is not hard to see that $(\iota \downarrow \mathrm{const}_x)$ is a coslice category whose objects are morphisms in $\D$ ending at $x$. In this situation, the identity morphism on $x$ is the initial object of the coslice category. It follows that $(\mathrm{Ran}_{\iota} W)_x \cong W_x$ for $x \in \Ob(\D)$. Consequently, $\res \circ \mathrm{Ran}_{\iota}$ is isomorphic to the identity functor on $\mathcal{O}_{\D} \Mod$. By \cite[Proposition III.2.5]{Gab}, we conclude that
\[
\mathcal{O}_{\D} \Mod \simeq \mathcal{O} \Mod / \ker (\res) = \mathcal{O} \Mod / \mathcal{T}(J) \simeq \Sh(\BC^{\op}, \, \mathcal{O})
\]
since the kernel of $\res$ is precisely $\mathcal{T}(J)$ by (1).
\end{proof}

\begin{remark} \label{comparison lemma}
Note that the above result holds for any rigid Grothendieck topology $J$ on any small category $\C^{\op}$ by letting $\D$ be the full subcategory of $\C$ consisting of $J$-irreducible objects; see \cite[C2.2, Theorem 2.2.3]{Jo} and \cite[C2.2, Example 2.2.4(d)]{Jo}.
\end{remark}

\section{Grothendieck topologies on certain noetherian EI categories}

In the previous section, we classified all Grothendieck topologies on $\C^{\op}$ under the assumption that $\C$ is a noetherian EI category. It is natural to consider the analogous question of classifying Grothendieck topologies on $\C^{\op}$ when $\C$ is an artinian EI category. In this case, the situation becomes substantially more complicated, and we need to impose additional assumptions on $\C$ besides the artinian condition. Throughout this section, let $\C$ be an EI category satisfying the following conditions:
\begin{itemize}
\item objects in $\C$ are parameterized by $\N$, so by abuse of notation we denote them by $n$, $n \in \N$;

\item $\C(m, n) \neq \varnothing$ if and only if $m \leqslant n$;

\item $\C(n, n)$ acts transitively on $\C(m, n)$ for every pair of objects $m$ and $n$;

\item $\C$ is a graded category equipped with a rank function $\deg: \mathrm{Mor}(\C) \to \N$ such that $\deg(f) = n - m$ for every morphism $f: m \to n$;

\item $\mathrm{Mor}(\C)$ is generated in degrees 0 and 1, namely every $f: m \to n$ with $n > m$ can be written as a composite of $n - m$ morphisms of degree 1.
\end{itemize}
We call them \textit{EI categories of type $\N$}. Typical examples include the following categories:
\begin{itemize}
\item the linearly ordered set $(\mathbb{N},\leqslant)$, viewed as a category whose objects are natural numbers and with a unique morphism $m\to n$ whenever $m \leqslant n$,

\item skeletal full subcategories of the category $\FI$ of finite sets and injections,

\item skeletal full subcategories of the category $\mathrm{VI}(\mathbbm{k})$ of finite dimensional vector spaces over a field $\mathbbm{k}$ and linear injections,

\item skeletal full subcategories of the category of finite cyclic $p$-groups and injective homomorphisms, where $p$ is a prime.
\end{itemize}
It is clear that $\C$ is an artinian EI category. Our goal is to classify Grothendieck topologies on $\Co$.

\begin{lemma}
For $m \in \N$, every nonempty subfunctor $S \subseteq \C(m, -)$ is principal; that is, $S$ is generated by a single morphism $f: m \to n$ in $\C$.
\end{lemma}

\begin{proof}
Choose a morphism $f: m \to n$ in $S$ such that $\deg(f)$ is minimal. We claim that
\[
S = \{gf \mid g \in \C(n, -) \}.
\]
Clearly, the right side is a subset of $S$. On the other hand, for every $h: m \to r$ in $S$, one has $\deg(h) \geqslant \deg(f)$, so $r \geqslant n$. Therefore, one can decompose $h$ as follows:
\[
m \overset{h_1} \longrightarrow m+1 \overset{h_2} \longrightarrow \ldots \overset{h_{n-m}} \longrightarrow n \overset{h_{n-m+1}} \longrightarrow \ldots \overset{h_{r-m}} \longrightarrow r
\]
Since $\C(n, n)$ acts transitively on $\C(m, n)$, we can find a certain $\delta \in \C(n, n)$ such that
\[
h = (h_{r-m} \ldots h_{n-m+1}) \circ (h_{n-m} \ldots h_1) = (h_{r-m} \ldots h_{n-m+1}) \circ (\delta \circ f)
\]
contained in $\{gf \mid g \in \C(n, -) \}$, so the equality holds.
\end{proof}

By the transitive action condition, if $S \subseteq \C(n, -)$ contains a morphism of degree $d$, then it contains all morphisms from $n$ to $n+d$. This observation as well as the above lemma tells us that nonempty subfunctors of $\C(n, -)$ are parameterized by natural numbers. Explicitly, for $r \in \N$, there is a unique subfunctor $S(n, r) \subseteq \C(n, -)$ consisting of all morphisms $f: n \to \bullet$ with $\deg(f) \geqslant r$. Moreover, one has the following inclusions:
\[
\C(n, -) = S(n, 0) \supsetneqq S(n, 1) \supsetneqq S(n, 2) \supsetneqq \ldots.
\]

Let $J$ be a Grothendieck topology on $\Co$ such that each $J(n)$ does not contain the empty subfunctor. We call it a \textit{generic} Grothendieck topology on $\Co$. Define a function
\[
\boldsymbol{d}: \mathbb{N} \to \mathbb{N} \cup \{\infty\}, \quad n \mapsto d_n = \sup \{r \in \N \mid S(n, r) \in J(n) \}.
\]
Note that $d_n$ might be $\infty$. Moreover, $J$ is completely determined by $\boldsymbol{d}$; that is,
\[
J(n) = \{ S(n, r) \mid r \leqslant d_n \}
\]
since $J$ is closed under upward inclusions. Conversely, given such a function $\boldsymbol{d}$, we get a rule $J$ assigning $J(n)$ to $n \in \Ob(\C)$, though $J$ might not be a Grothendieck topology.

The above construction gives an injective map
\[
\{ \text{generic Grothendieck topology $J$ on } \Co \} \longrightarrow \{ \text{function } \boldsymbol{d}: \Ob(\C) = \N \to \N \cup \{\infty\} \}.
\]

\begin{lemma}
Let $f: m \to n$ be a morphism in $\C$. Then one has
\[
f^{\ast}(S(m, r)) = \begin{cases}
S(n, r - \deg(f)) = S(n, r + m - n), & \text{if } r+m \geqslant n;\\
\C(n, -), & \text{else.}
\end{cases}
\]
\end{lemma}

\begin{proof}
Recall that $S(m,r)$ consists of all morphisms $h:m\to \bullet$ with $\deg(h) \geqslant r$. By definition of the pullback of a sieve, one has
\[
f^{\ast}(S(m,r)) = \{g:n\to \bullet \mid gf\in S(m,r)\}.
\]
Since $\deg(gf) = \deg(g) + \deg(f)$, the condition $gf \in S(m,r)$ is equivalent to
\[
\deg(g) \geqslant r - \deg(f).
\]
If $r-\deg(f) \geqslant 0$, this means precisely that
\[
g \in S(n,r-\deg(f)),
\]
and hence
\[
f^{\ast}(S(m,r)) = S(n, r - \deg(f)) = S(n, r+m-n).
\]

If $r - \deg(f) < 0$, then every morphism $g: n \to \bullet$ satisfies $\deg(g) \geqslant r - \deg(f)$, so the pullback contains all morphisms from $n$. Therefore,
\[
f^{\ast}(S(m,r)) = \C(n,-).
\]
This proves the formula.
\end{proof}

The following result classifies all generic Grothendieck topologies on $\Co$.

\begin{proposition} \label{generic topologies}
Generic Grothendieck topologies on $\Co$ are parameterized by functions
\[
\boldsymbol{d}: \Ob(\C) = \N \to \N \cup \{\infty\}
\]
satisfying the following condition: if $d_n \neq 0$, then $d_{n+1} = d_n - 1$.
\end{proposition}

More explicitly, a sequence $\boldsymbol{d}: \N \to \N \cup \{ \infty \}$ satisfies the above condition if and only if it decomposes into the following pieces:
\begin{enumerate}
\item $(0, 0, 0, \ldots, 0)$;

\item $(r, r - 1, r - 2, \ldots, 1, 0)$;

\item $(\infty, \infty, \infty, \ldots)$,
\end{enumerate}
where the last piece might or might not appear. But if it appears, it must be the tail of the sequence.

For instance, if all pieces of $\boldsymbol{d}$ are of the first form, then it corresponds to the minimal topology; if $\boldsymbol{d}$ has only one piece which is of the third form, then it corresponds to the atomic topology.

\begin{proof}
First we check that the rule $J$ determined by a function $\boldsymbol{d}$ satisfying the specified condition is indeed a Grothendieck topology on $\Co$ (it is clearly generic by the construction). The maximal axiom is clear since $\C(n, -) \in J(n)$, so we check the stability axiom and the transitivity axiom.

\textbf{Stability.} Take $m \in \Ob(\C)$, $S(m, r) \in J(m)$ with $r \leqslant d_m$, and $f \in \C(m ,n)$. If $d_n = \infty$, then $f^{\ast} (S(m, r)) \in J(n)$ since $J(n)$ contains all nonempty subfunctors of $\C(n, -)$. Thus we may assume $d_n \neq \infty$ and hence $d_m \neq \infty$ as well. We have two cases:
\begin{itemize}
\item If $r + m \geqslant n$, then $d_m + m \geqslant n$, so $d_n = d_m + m - n \geqslant 0$. By the previous lemma,
\[
f^{\ast} (S(m, r)) = S(n, m + r - n) \in J(n)
\]
since $m+r-n \leqslant d_m + m -n = d_n$.

\item If $r + m < n$, then $f \in S(m, r)$ and $f^{\ast} (S(m, r)) = \C(n, -)$, which is clearly contained in $J(n)$.
\end{itemize}

\textbf{Transitivity.} Take $m \in \Ob(\C)$, $S(m, s) \in J(m)$ with $s \leqslant d_m$, and $S(m, r) \subseteq \C(m, -)$. Suppose that $f^{\ast} (S(m, r)) \in J(n)$ for $f: m \to n$ in $S(m, s)$. We want to show $S(m, r) \in J(m)$; in other words, $r \leqslant d_m$. This is clearly true if $d_m = \infty$, so we assume that $d_m$ is finite. We have two cases:
\begin{enumerate}
\item If $d_m = 0$, then $J(m) = \{ \C(m, -) \}$ and $S(m, s) = \C(m, -)$. In this case, one can take $f$ to be the identity morphism on $m$. Therefore,
\[
S(m, r) = f^{\ast} (S(m, r)) \in J(m)
\]
implies $S(m, r) = \C(m, -)$, so $r = 0$ and the conclusion holds.

\item Suppose that $d_m > 0$. In this case, if $r \leqslant s$, then $r \leqslant d_m$, so we can assume that $s < r$. For a morphism $f: m \to m+s$ in $S(m, s)$, by the previous lemma,
\[
f^{\ast} (S(m, r)) = S(m+s, r-s) \in J(m+s),
\]
so $r - s \leqslant d_{m+s}$. But since $d_m \geqslant s$, one knows that $d_{m+s} = d_m -s$ by the specified condition on $\boldsymbol{d}$. Therefore, we have $r - s \leqslant d_m -s$, so $r \leqslant d_m$, as desired.
\end{enumerate}

We have checked that the rule $J$ is indeed a generic Grothendieck topology on $\Co$. Conversely, given a generic Grothendieck topology $J$ on $\Co$, we want to show that the corresponding function $\boldsymbol{d}$ satisfies the specified condition. One has two cases:
\begin{enumerate}

\item If $d_m = \infty$, then $S(m, r) \in J(m)$ for $r \geqslant 0$. For $f \in \C(m, m+1)$ and $r \geqslant 1$, one has
\[
f^{\ast} (S(m, r)) = S(m+1, r-1) \in J(m+1),
\]
so $J(m+1)$ contains all nonempty subfunctors of $\C(m+1, -)$, namely $d_{m+1} = \infty = d_m - 1$.

\item If $d_m < \infty$, one has $d_{m+1} \geqslant d_m - 1$ by an argument similar to the previous case. If $d_{m+1} > d_m - 1$, then for $n = m + d_m$,
\[
d_n \geqslant d_{n-1} - 1 \geqslant \ldots \geqslant d_{m+1} - (n - m - 1) \geqslant d_m + m + 1 - n = 1,
\]
so $S(n, 1) \in J(n)$. Consequently, for $f: m \to l$ in $S(m, d_m) \in J(m)$, one has
\[
f^{\ast} (S(m, d_m+1)) = \begin{cases}
S(n, d_m + 1 - \deg(f)) = S(n, 1) \in J(n), & \text{ if } l = n;\\
S(l, 0) \in J(l), & \text{ if } l > n.
\end{cases}
\]
Therefore, $S(m, d_m + 1) \in J(m)$ by the transitivity axiom, so $d_m \geqslant d_m + 1$, which is impossible. Thus we also have $d_{m+1} = d_m - 1$.
\end{enumerate}
\end{proof}

\begin{corollary} \label{another parametrization}
Generic Grothendieck topologies on $\Co$ are parameterized by subsets of $\Ob(\C)$.
\end{corollary}

\begin{proof}
Given a function $f: \mathbb{N} \to \{0, \, 1\}$, we define a function $\boldsymbol{d}: \N \to \N \cup \{\infty \}$ by setting
\[
d_n = \sup \{ s+1 \mid f(n) = f(n+1) = \ldots = f(n+s) = 1 \}.
\]
Conversely, given a function $\boldsymbol{d}$, we define $f$ by setting $f(n) = 1$ when $d_n > 0$, and $f(n) = 0$ when $d_n = 0$. It is obvious that this construction gives a bijection between $2^{\mathbb{N}}$ and the set of functions $\boldsymbol{d}$ in the above proposition.
\end{proof}

\begin{remark}
Generic Grothendieck topologies on $\Co$, although parameterized by subsets of $\Ob(\C)$, need not be rigid. Indeed, the atomic Grothendieck topology $J_a$
is generic. However, no object is $J_a$-irreducible, and hence $J_a$ is not rigid.

Conversely, rigidity does not imply genericity. Consider the subcategory topology $K$ on $\Co$ such that no object is $K$-irreducible. This is the maximal Grothendieck topology, and hence it is rigid, but it is not generic since $K(x)$ contains the empty sieve for every object $x$.

We also observe that, in contrast to Lemma~\ref{intersection}, the atomic Grothendieck topology $J_a$ is not closed under arbitrary intersections. Indeed,
for $x\in\Ob(\C)$, the intersection of all $S\in J_a(x)$ is the empty sieve, which does not belong to $J_a(x)$.
\end{remark}

Now we consider non-generic Grothendieck topologies $J$ on $\Co$. In this case there exists a certain object $m$ such that $\varnothing \in J(m)$. Here are some elementary observations:
\begin{itemize}
\item By the stability axiom, $\varnothing \in J(n)$ for all $n \geqslant m$.

\item If $m \geqslant 1$ and $\varnothing \notin J(m-1)$, then $J(m-1)$ only contains the subfunctor $\C(m-1, -)$. Indeed, if $S(m-1, r) \in J(m-1)$ for a certain $r \geqslant 1$, then $\varnothing \in J(m-1)$ via applying the transitivity axiom to $S(m-1, r) \in J(m-1)$ and $\varnothing \subseteq \C(m-1, -)$.

\item There are no objects $n$ such that $J(n)$ consists of all nonempty subfunctors of $\C(n, -)$. Indeed, by the first observation, we shall have $n < m$. But applying the transitivity axiom to $\varnothing \subseteq \C(n, -)$ and the subfunctor $S(n, m-n) \in J(n)$, we deduce that $\varnothing \in J(n)$, which is impossible.
\end{itemize}

These observations allow us to define a function
\[
\{ \text{non-generic Grothendieck topology $J$ on } \Co \} \longrightarrow \{ \text{function } \boldsymbol{d}: \Ob(\C) = \N \to \N \cup \{-\infty\} \}
\]
in the following way: if $J(n)$ contains the empty functor, then $d_n = - \infty$; otherwise, $J(n)$ is a finite set by the third observation, so we define $d_n$ as in the generic case.

The following result classifies all non-generic Grothendieck topologies on $\Co$.

\begin{proposition} \label{nongeneric topologies}
Non-generic Grothendieck topologies on $\Co$ are parameterized by functions
\[
\boldsymbol{d}: \Ob(\C) = \N \to \N \cup \{-\infty\}
\]
satisfying the following conditions: $d_n = -\infty$ for $n \gg 0$, and if $d_n \neq 0$, then $d_{n+1} = d_n - 1$.
\end{proposition}

More explicitly, a sequence $\boldsymbol{d}: \N \to \N \cup \{ -\infty \}$ satisfies the above condition if and only if it decomposes into the following pieces:
\begin{enumerate}
\item $(0, 0, 0, \ldots, 0)$;

\item $(r, r - 1, r - 2, \ldots, 1, 0)$;

\item $(-\infty, -\infty, -\infty, \ldots)$,
\end{enumerate}
where the last piece must appear and must be the tail of the sequence.

\begin{proof}
The proof is similar to that of Proposition \ref{generic topologies}, so we omit it.
\end{proof}

\begin{corollary}
Non-generic Grothendieck topologies on $\Co$ are parameterized by finite subsets of $\Ob(\C)$.
\end{corollary}

\begin{proof}
Note that there is a natural bijection between the set of finite subsets of $\mathbb{N}$ and the set of functions $\{f: \mathbb{N} \to \{0, \, 1\} \mid f(n) = 0 \text{ for } n \gg 0 \}$. With this observation, the proof is similar to that of Corollary \ref{another parametrization}.
\end{proof}

The following result classifies all rigid Grothendieck topologies on $\Co$.

\begin{proposition} \label{rigid topologies for type n}
A Grothendieck topology $J$ on $\C^{\op}$ is rigid if and only if its corresponding function $\boldsymbol{d}$ satisfies the following condition: $d_n \neq \infty$ for all $n \in \N$. In particular, every non-generic Grothendieck topology is rigid.
\end{proposition}

In other words, the function $\boldsymbol{d}$ can be represented by a sequence which decomposes into the following pieces:
\begin{enumerate}
\item $(0, 0, 0, \ldots, 0)$;

\item $(r, r - 1, r - 2, \ldots, 1, 0)$;

\item $(-\infty, -\infty, -\infty, \ldots)$,
\end{enumerate}
where the last piece may or may not appear. But if it appears, it must be the tail of the sequence.

\begin{proof}
It is routine to check that Grothendieck topologies satisfying this condition are rigid. To prove the only if direction, we need to show that if $d_n = \infty$ for a certain $n \in \N$, then the corresponding Grothendieck topology on $\Co$ is not rigid. Let $n$ be the minimal number such that $d_n = \infty$. Note that if $m \geqslant n$, then $d_m = \infty$, and hence $m$ is not $J$-irreducible. Thus the subfunctor of $\C(n, -)$ generated by morphisms from $n$ to $J$-irreducible objects is the empty set, which is not contained in $J(n)$. Consequently, $J$ is not rigid.
\end{proof}

We call a Grothendieck topology $J$ on $\C^{\op}$ \textit{exceptional} if it is not rigid, which shall be generic by the previous proposition. Therefore, its corresponding function $\boldsymbol{d}$ satisfies the following condition: $d_n = \infty$ for some $n \in \N$. Since an explicit description of $\Sh(\BC^{\op}, \mathcal{O})$ for rigid Grothendieck topologies has been obtained, in the rest of this section we give a description of $\Sh(\BC^{\op}, \mathcal{O})$ for exceptional Grothendieck topologies.

Let $J$ be an exceptional Grothendieck topology on $\Co$ corresponding to a function $\boldsymbol{d}$, and let $\D$ be the full subcategory consisting of objects $n$ such that $d_n = 0$ or $d_n = \infty$. We first verify that $\D^{\op}$ is a $J$-dense subcategory of $\C^{\op}$ (for a definition of $J$-dense subcategories, see \cite[C2.2, Definition 2.2.1]{Jo}).

For an object $n$, let $S_{\D}(n) \subseteq \C(n,-)$ be the sieve generated by all morphisms from $n$ to objects of $\D$. We claim that $S_{\D}(n)\in J(n)$. If $d_n$ is $\infty$ or $0$, then $n$ is an object in $\D$, and hence $S_{\D}(n) = \C(n,-) \in J(n)$. Suppose that $d_n = r$ is a positive integer. By the characterization of generic Grothendieck topologies, we have
\[
d_{n+1}=r-1,\quad d_{n+2}=r-2,\quad \ldots,\quad d_{n+r}=0.
\]
In particular, $n+r$ is an object in $\D$. The sieve generated by the morphisms $n\to n+r$ is precisely $S(n,r)$, since every morphism of degree at least
$r$ factors through an object of degree $n+r$. Hence
\[
S(n,r)\subseteq S_{\D}(n).
\]
On the other hand, $S(n,r)\in J(n)$ by the definition of $d_n=r$. Since $J(n)$ is closed under upward inclusions, we obtain $S_{\D}(n) \in J(n)$. Therefore, $\D^{\op}$ is $J$-dense.

By the $J$-density of $\D^{\op}$, restriction induces an equivalence of sheaf categories, as in the proof of Theorem \ref{classify sheaves}. Thus one has
\[
\Sh(\BC^{\op}, \mathcal{O}) \simeq \Sh(\boldsymbol{\mathscr{D}}^{\op}, \mathcal{O}_{\D}),
\]
where $\boldsymbol{\mathscr{D}}^{\op} = (\D^{\op}, J_{\D})$, $J_{\D}$ is a Grothendieck topology on $\D^{\op}$ such that
\[
J_{\D}(n) =
\begin{cases}
\{ \D(n, -) \}, &  \text{if } d_n = 0;\\
\{ S \mid S \subseteq \D(n, -) = \C(n, -) \}, & \text{if } d_n = \infty,
\end{cases}
\]
and $\mathcal{O}_{\D}$ is the presheaf of rings on $\D^{\op}$ obtained via restricting $\mathcal{O}$ to $\D^{\op}$. The full subcategory $\D$ is an artinian category whose objects can still be parameterized by $\mathbb{N}$, and $J_{\D}$ corresponds to the sequence $(0, 0, \ldots, 0, \infty, \infty, \ldots)$.

Let $\mathscr{E}$ be the full subcategory of $\D$ consisting of objects $n$ such that $d_n = \infty$. It is an ideal of $\D$, so the restriction of $J_{\D}$ to $\mathscr{E}^{\op}$ defines a Grothendieck topology $J_{\mathscr{E}}$ on $\mathscr{E}^{\op}$ by Lemma \ref{topologies on coideals}, and the restriction $\mathcal{O}_{\mathscr{E}}$ of $\mathcal{O}_{\D}$ to $\mathscr{E}^{\op}$ is a structure sheaf on $(\mathscr{E}^{\op}, J_{\mathscr{E}})$. Note that $J_{\mathscr{E}}$ is the atomic Grothendieck topology.

\begin{proposition} \label{sheaves for type n}
An $\mathcal{O}_{\D}$-module $V$ is a sheaf of modules on $(\boldsymbol{\mathscr{D}}^{\op}, \mathcal{O}_{\D})$ if and only if its restriction $V_{\mathscr{E}}$ to $\mathscr{E}^{\op}$ is a sheaf of modules on $(\boldsymbol{\mathscr{E}}^{\op}, \mathcal{O}_{\mathscr{E}})$.
\end{proposition}

\begin{proof}
This can be checked directly by Definition \ref{def of sheaves}. We give a representation theoretic proof. By Corollary \ref{another equivalence}, $V$ is a sheaf of modules on $(\boldsymbol{\mathscr{D}}^{\op}, \mathcal{O}_{\D})$ if and only if the inclusion $\mathbb{Z}S \to \mathbb{Z}\D(x, -)$ induces an isomorphism
\begin{equation}
\Hom_{\underline{\mathbb{Z}}_{\D} \Mod} (\mathbb{Z}\D(x, -), V) \cong \Hom_{\underline{\mathbb{Z}}_{\D} \Mod} (\mathbb{Z}S, V),
\end{equation}
for each $x \in \Ob(\D)$ and each $S \in J_{\D}(x)$.  If $x \in \Ob(\D) \setminus \Ob(\mathscr{E})$, then $S = \D(x, -)$, so the above isomorphism automatically holds. When $x \in \Ob(\mathscr{E})$, since $\mathscr{E}$ is an ideal of $\D$, we deduce that $J_{\D}(x) = J_{\mathscr{E}}(x)$, $\D(x, -) = \mathscr{E}(x, -)$ (viewed as sets of morphisms), and the inclusion functor $\mathscr{E} \to \D$ induces the following natural isomorphism
\[
\Hom_{\underline{\mathbb{Z}}_{\D} \Mod} (\mathbb{Z}S, V) \cong \Hom_{\underline{\mathbb{Z}}_{\mathscr{E}} \Mod} (\mathbb{Z}S, V_{\mathscr{E}})
\]
for $S \in J_{\D}(x)$. Thus formula (6.1) holds if and only if
\[
\Hom_{\underline{\mathbb{Z}}_{\mathscr{E}} \Mod} (\mathbb{Z}\mathscr{E}(x, -), V) \cong \Hom_{\underline{\mathbb{Z}}_{\mathscr{E}} \Mod} (\mathbb{Z}S, V_{\mathscr{E}})
\]
for each $x \in \Ob(\mathscr{E})$ and each $S \in J_{\mathscr{E}} (x)$, and if and only if $V_{\mathscr{E}}$ is a sheaf of modules on $(\boldsymbol{\mathscr{E}}^{\op}, \mathcal{O}_{\mathscr{E}})$.
\end{proof}

Consequently, to classify sheaves of modules on all ringed sites $(\BC^{\op}, \mathcal{O})$, it is enough to classify sheaves of modules on the ringed atomic site $(\boldsymbol{\mathscr{E}}^{\op}, \mathcal{O}_{\mathscr{E}})$. However, this still seems to be a hard problem. For some examples, please refer to \cite{DLLX}.

\begin{remark} \label{type Z}
Replacing the first condition in the definition of EI categories of type $\N$ by the following one: objects in $\C$ are parameterized by $n \in \mathbb{Z}$, we can define \textit{EI categories of type $\mathbb{Z}$}. The reader can mimic our approach to classify Grothendieck topologies on $\Co$ where $\C$ is an EI category of type $\mathbb{Z}$. Explicitly, we have:
\begin{enumerate}
\item Generic Grothendieck topologies on $\Co$ are parameterized by functions
\[
\boldsymbol{d}: \Ob(\C) = \mathbb{Z} \to \N \cup \{\infty\}
\]
satisfying the following condition: if $d_n \neq 0$, then $d_{n+1} = d_n - 1$.

\item Non-generic Grothendieck topologies on $\Co$ are parameterized by functions
\[
\boldsymbol{d}: \Ob(\C) = \mathbb{Z} \to \N \cup \{-\infty\}
\]
satisfying the following conditions: $d_n = -\infty$ for $n \gg 0$; and if $d_n \neq 0$, then $d_{n+1} = d_n - 1$.

\item A Grothendieck topology $J$ on $\Co$ is rigid if and only if the corresponding function $\boldsymbol{d}$ satisfies the following condition: $d_n \neq \infty$ for all $n \in \mathbb{Z}$.
\end{enumerate}
One can deduce the above results from Lemma \ref{topologies on coideals}. Indeed, let $J$ be a Grothendieck topology on $\Co$. Then for each $n \in \mathbb{Z}$, the full subcategory $\C_{\geqslant n}$ of objects $m \geqslant n$ is an ideal of $\C$ and an EI category of type $\N$. These results follow via applying Propositions \ref{generic topologies}, \ref{nongeneric topologies} and \ref{rigid topologies for type n} to the restriction of $J$ to $\C_{\geqslant n}$. When $\C$ is the poset $\mathbb{Z}$, they have been proved in \cite[Proposition B.31]{Lin}.
\end{remark}

\section{Applications}

In this section we describe some applications of our results in representation theory of combinatorial categories and groups.

\subsection{Representations of infinite full subcategories of $\FI$ and $\VI_q$}

Throughout this subsection let $k$ be a commutative noetherian ring, let $\C$ be a skeleton of the category $\FI$ of finite sets and injections, or the category $\VI_q$ of finite dimensional vector spaces over a finite field $\mathbb{F}_q$ and $\mathbb{F}_q$-linear injections, and let $\D$ be a full subcategory of $\C$ with infinitely many objects. Denote objects in $\C$ by $\boldsymbol{n}$, $n \in \mathbb{N}$, which is isomorphic to $[n] = \{1, 2, \ldots, n\}$ for $\FI$ or $\mathbb{F}_q^n$ for $\VI_q$. Accordingly, let $P(n) $ be $k\C(\boldsymbol{n}, -)$.

One may impose a rigid Grothendieck topology $J_r$ on $\C^{\op}$ such that $J_r$-irreducible objects in $\C$ are precisely objects in $\D$. On the other hand, it is easy to check that $\C^{\op}$ satisfies the right Ore condition, so one can also impose the atomic Grothendieck topology $J_a$ on it.

\begin{lemma} \label{canonical}
Notation as above. For each $n \in \mathbb{N}$, $P(n)$ is a sheaf of modules on the ringed site $(\C^{\op}, \, J_a, \, \underline{k})$, where $\underline{k}$ is the constant structure sheaf, and hence is a sheaf of modules on the ringed site $(\C^{\op}, \, J_r, \, \underline{k})$.
\end{lemma}

\begin{proof}
Firstly, we show that $J_r(n) \subseteq J_a(n)$ for each $n \in \mathbb{N}$. If this is true, then every sheaf of modules over $(\C^{\op}, \, J_a, \, \underline{k})$ is automatically a sheaf of modules over $(\C^{\op}, \, J_r, \, \underline{k})$, and the second statement follows from this fact and the first statement.

By the definition of atomic topology, $J_r(n) \subseteq J_a(n)$ is true if and only if $J_r(n)$ does not contain the empty sieve $\varnothing$. Suppose that, for contradiction, $\varnothing \in J_r(n)$ for a certain $n \in \mathbb{N}$. Then by Proposition \ref{nongeneric topologies}, $J_r(m)$ contains the empty sieve for all $m \geqslant n$. Thus each object $\boldsymbol{m}$ with $m \geqslant n$ is not $J_r$-irreducible, and hence not an object of $\D$. Consequently, $\D$ can only have finitely many objects, contradicting the assumption that $\D$ has infinitely many objects. Thus $J_r(n) \subseteq J_a(n)$ for every $n \in \mathbb{N}$ as desired.

Now we turn to the first statement. By Artin's theorem \cite[III. 9, Theorems 1 and 2]{MM}, there is an equivalence
\[
\Sh(\C^{\op}, \, J_a, \, \underline{k}) \simeq kG \dMod,
\]
where $kG \dMod$ is the category of \textit{discrete representations} of $G$, and $G$ is the infinite symmetric group $\varinjlim_n \Sigma_n$ or the infinite general linear group $\varinjlim_n \mathrm{GL}_n (\mathbb{F}_q)$. An explicit computation shows that $P(n)$ is the sheaf corresponding to the discrete $kG$-module $k(G/H_n)$, where $H_n$ is the stabilizer subgroup of $G$ fixing every element in $\boldsymbol{n} = [n]$ or $\boldsymbol{n} = \mathbb{F}_q^n$. For details, see \cite[Section 5]{DLLX}.
\end{proof}

\begin{remark}
One can also prove the conclusion by showing that $P(n)$ is $J_a$-saturated. For $\FI$, this is established in \cite[Theorem C]{LR}; for $\VI_q$, under the extra assumption that $q$ is invertible in $k$, this follows from \cite[Corollary 4.22]{Nag1}.
\end{remark}

It follows from this lemma that both $J_r$ and $J_a$ are \textit{subcanonical}; that is, representable presheaves are sheaves. Therefore, the following set is a family of generators of $\Sh(\C^{\op}, \, J_a, \, \underline{k})$:
\[
\{ P(n) \mid n \in \mathbb{N} \},
\]
so we can make the following definition.

\begin{definition}
We say that $V \in \Sh(\C^{\op}, \, J_a, \, \underline{k})$ is \textit{finitely generated} if there is a surjection
\[
\bigoplus_n P(n)^{c_n} \twoheadrightarrow V
\]
such that the sum of all multiplicities $c_n$ is finite. Finitely generated sheaves in $\Sh(\C^{\op}, \, J_r, \, \underline{k})$ are defined similarly.
\end{definition}

In the rest of this subsection we will frequently use the following easy observation: a sheaf homomorphism $V \to W$ is injective (resp., surjective) if and only if it is injective (resp., surjective) as a presheaf homomorphism. Consequently, $V \in \Sh(\C^{\op}, \, J_a, \, \underline{k})$ is finitely generated if and only if viewed as a $\C$-module it is finitely generated. The same conclusion holds for $\Sh(\C^{\op}, \, J_r, \, \underline{k})$.

\begin{proposition} \label{noetherianity}
Every finitely generated $\D$-module over $k$ is noetherian.
\end{proposition}

\begin{proof}
It suffices to prove the following conclusion: for each $\boldsymbol{n} \in \Ob(\D)$, every submodule $V$ of $Q(n) = k\D(\boldsymbol{n}, -)$ is finitely generated. By \cite[C2.2, Theorem 2.2.3]{Jo} and Remark \ref{comparison lemma}, one has
\[
\Sh(\C^{\op}, \, J_r, \, \underline{k}) \simeq \D \Mod,
\]
for which an explicit equivalence is given by the restriction functor $\res$ from the sheaf category to the module category and the coinduction functor $\mathrm{coind}$ (the right Kan extension) in the other direction. Accordingly, $\coind(V)$ is a subsheaf of $\coind(Q(n))$ in $\Sh(\C^{\op}, \, J_r, \, \underline{k})$. Since $P(n) \in \Sh(\C^{\op}, \, J_r, \, \underline{k})$ by Lemma \ref{canonical} and $\res (P(n)) = Q(n)$, it follows that $\coind (Q(n)) \cong P(n)$. Thus without loss of generality we can assume that $\coind(V)$ is a subsheaf of $P(n)$.

By the noetherianity results in \cite{CEFN, SS1}, $\coind(V)$ is a finitely generated $\C$-module, so one has a surjection
\[
\bigoplus_{i \in \mathbb{N}} P(i)^{c_i} \longrightarrow \coind(V)
\]
of $\C$-modules such that the sum of these $c_i$'s is finite. By the above observation, this is also a surjection in $\Sh(\C^{\op}, \, J_r, \, \underline{k})$. Applying the restriction functor, we obtain a surjection
\[
\bigoplus_{i \in \mathbb{N}} \res(P(i))^{c_i} \longrightarrow \res(\coind(V)) \cong V
\]
of $\D$-modules. It remains to show that each $\res(P(i))$ is a finitely generated $\D$-module.

In the case $\boldsymbol{i} \in \Ob(\D)$, one has $\res(P(i)) = Q(i)$, so the conclusion holds trivially. In the case $\boldsymbol{i} \notin \Ob(\D)$, let $\boldsymbol{j}$ be the minimal number such that $j > i$ and $\boldsymbol{j} \in \Ob(\D)$, which always exists since $\D$ is an infinite subcategory of $\C$. Then the $\D$-module $\res(P(i))$ is generated by its value on $\boldsymbol{j}$, and hence is finitely generated as well. The proof is completed.
\end{proof}

Let $\sh(\D^{\op}, \, J_a, \, \underline{k})$ be the category of finitely generated objects in $\Sh(\D^{\op}, \, J_a, \, \underline{k})$. The above proposition tells us that $\sh(\D^{\op}, \, J_a, \, \underline{k})$ is an abelian category; for details, see the comment before \cite[Remark 3.15]{DLLX}.

In the rest of this subsection let $k$ be a field of characteristic 0.

\begin{proposition} \label{self-injectivity}
Every indecomposable projective $\D$-module over $k$ is also injective.
\end{proposition}

\begin{proof}
It suffices to show that $Q(n) = k\D(\boldsymbol{n}, -)$ is injective for each $\boldsymbol{n} \in \Ob(\D)$. Using the equivalence from $\D \Mod$ to $\Sh(\C^{\op}, \, J_r, \, \underline{k})$ given by the coinduction functor, one only needs to show the injectivity of $P(n) \cong \coind(Q(n))$ in $\Sh(\C^{\op}, \, J_r, \, \underline{k})$. By Corollary \ref{injective sheaves}, this is equivalent to saying that $P(n)$ is a $J_r$-torsion free injective $\C$-module. But it has been proved in \cite{GL} that $P(n)$ is an injective $\C$-module, so it remains to show that $P(n)$ is $J_r$-torsion free. This is clear. Indeed, $P(n)$ is $J_a$-torsion free as every morphism in $\C$ is monic, so it is $J_r$-torsion free as well by Lemma \ref{correspondence}. This finishes the proof.
\end{proof}

To prove other statements in Theorem \ref{Theorem 6}, we turn to the atomic Grothendieck topology $J_a$ on $\C^{\op}$. Since $\D$ has infinitely many objects, $\D^{\op}$ is a \textit{$J_a$-dense subcategory} in the sense of \cite[C2.2, Definition 2.2.1]{Jo}. Indeed, the atomic topology $J_a$ consists of all nonempty sieves. Since $\C$ is of type $\N$ and $\D$ has infinitely many objects, the objects of $\D$ are unbounded in $\N$. Hence, for every object $\boldsymbol{n}$ in $\C$, there exists an object $\boldsymbol{m}$ in$ \D$ with $n \leqslant m$, and therefore a morphism $\boldsymbol{n} \to \boldsymbol{m}$. Consequently, the sieve generated by morphisms from $\boldsymbol{n}$ to objects of $\D$ is nonempty, and thus belongs to $J_a(\boldsymbol{n})$. Hence $\D^{\op}$ is $J_a$-dense. Furthermore, the \textit{induced} Grothendieck topology $K$ on $\D^{\op}$ via setting
\[
K(n) = \{ S \cap \D(\boldsymbol{n}, -) \mid S \in J_a(\boldsymbol{n})\}
\]
coincides with the atomic Grothendieck topology. Thus by abuse of notation we also use $J_a$ to denote this induced Grothendieck topology. Applying \cite[C2.2, Theorem 2.2.3]{Jo}, we obtain an immediate result:

\begin{lemma} \label{equivalence of atomic sheaves}
There is an equivalence
\[
\Sh(\C^{\op}, \, J_a, \, \underline{k}) \simeq \Sh(\D^{\op}, \, J_a, \, \underline{k})
\]
given by the restriction functor and the coinduction functor.
\end{lemma}

\begin{remark}
Artin's theorem provides another viewpoint to interpret this equivalence. For instance, if $G = \varinjlim_n \Sigma_n$, since both
\[
\{H_n = \{g \in G \mid g(i) = i, \, \forall i \in [n] \} \mid \boldsymbol{n} \in \Ob(\C) \}
\]
and
\[
\{H_n = \{g \in G \mid g(i) = i, \, \forall i \in [n] \} \mid \boldsymbol{n} \in \Ob(\D) \}
\]
are cofinal systems of open subgroups of $G$, one has
\[
\Sh(\C^{\op}, \, J_a, \, \underline{k}) \simeq kG \dMod \simeq \Sh(\D^{\op}, \, J_a, \, \underline{k}).
\]
For details, please refer to \cite[Section 5]{DLLX}.

The above equivalence gives another proof of Proposition \ref{self-injectivity}. For $\boldsymbol{n} \in \Ob(\D)$, since $P(n)$ is a $J_a$-torsion free injective $\C$-module, it is also injective in $\Sh(\C^{\op}, \, J_a, \, \underline{k})$ by Corollary \ref{injective sheaves}. Applying the restriction functor, we deduce that $Q(n) \cong \res (P(n))$ is injective in $\Sh(\D^{\op}, \, J_a, \, \underline{k})$. Again, by Corollary \ref{injective sheaves}, $Q(n)$ is an injective $\D$-module.
\end{remark}

\begin{lemma} \label{embedding}
Let $V$ be a finitely generated $J_a$-torsion free $\D$-module. Then there is an injection $V \to \res(P)$ with $P$ a finitely generated projective $\C$-module.
\end{lemma}

\begin{proof}
Since $V$ is $J_a$-torsion free, one may obtain an injection $V \to V^{\sharp}$ in $\D \Mod$, where $V^{\sharp}$ is the sheafification of $V$ with respect to $J_a$. Thus without loss of generality we can assume that $V$ is contained in $\Sh(\D^{\op}, \, J_a, \, \underline{k})$. Accordingly, $\coind(V)$ is contained in $\Sh(\C^{\op}, \, J_a, \, \underline{k})$.

By \cite[Proposition 7.5]{GL} and \cite[Theorem 4.34]{Nag1}, there is an injection $\coind(V) \to P$ in $\C \Mod$, where $P$ is a finitely generated projective $\C$-module. This is also an injection in $\Sh(\C^{\op}, \, J_a, \, \underline{k})$ by the above mentioned observation. By Lemma \ref{equivalence of atomic sheaves}, one obtains an injection $V \to \res(P)$ in $\Sh(\D^{\op}, \, J_a, \, \underline{k})$, which is the desired injection in $\D \Mod$.
\end{proof}

Now we finish the proof of Theorem \ref{Theorem 6}.

\begin{proposition} \label{generalized representation stability}
Let $k$ be a field of characteristic 0. Then:
\begin{enumerate}
\item the category $\D \module$ of finitely generated $\D$-modules has enough injectives;

\item up to isomorphism, an indecomposable injective $\D$-module is one of the following three:
\begin{itemize}
\item a summand of $k\D(\boldsymbol{n}, -)$ for a certain $\boldsymbol{n} \in \Ob(\D)$,
\item a summand of the restriction of $k\C(\boldsymbol{n}, -)$ for a certain $\boldsymbol{n} \in \Ob(\C) \setminus \Ob(\D)$,
\item a finite dimensional injective $\D$-module;
\end{itemize}

\item one has the following equivalences
\[
\D \module / \D \fdmod \simeq \C \fdmod.
\]

\item every finitely generated $\D$-module has finite injective dimension.
\end{enumerate}
\end{proposition}

\begin{proof}
(1). A finitely generated $\D$-module $V$ gives a short exact sequence
\[
0 \to V_T \to V \to V_F \to 0
\]
in $\D \Mod$, where $V_T$ is $J_a$-torsion and $V_F$ is $J_a$-torsion free. Since $V_T$ is finitely generated, it is finite dimensional. Hence, $V_T$ can be embedded into a finite dimensional injective $\D$-module. By Lemma \ref{embedding}, $V_F$ can be embedded into $\res(P)$ with $P$ a finitely generated projective $\C$-module. But $P$ is injective in $\Sh(\C^{\op}, \, J_a, \, \underline{k})$, so $\res(P)$ is injective in $\Sh(\D^{\op}, \, J_a, \, \underline{k})$, and hence injective in $\D \Mod$.

(2). Let $V$ be an indecomposable injective $\D$-module. By (1), it can be embedded into a direct sum of a finite dimensional injective $\D$-module and $\res(P)$, so it is either finite dimensional or a direct summand of $\res(P)$. In the latter case, it must be a direct summand of some $\res(P(n))$. The conclusion then follows.

(3). By \cite[Corollary 4.9]{DLLX}, one has
\[
\sh(\C^{\op}, \, J_a, \, \underline{k}) \simeq \C \module / \C \fdmod \simeq \C \fdmod
\]
The finitely generated version of \cite[Corollary 3.9]{DLLX} gives another equivalence
\[
\sh(\D^{\op}, \, J_a, \, \underline{k}) \simeq \D \module / \D \fdmod.
\]
But it is clear that the equivalence in Lemma \ref{equivalence of atomic sheaves} restricts to an equivalence
\[
\sh(\C^{\op}, \, J_a, \, \underline{k}) \simeq \sh(\D^{\op}, \, J_a, \, \underline{k}).
\]
The conclusion then follows.

(4). We can use the same technique as that of the proof of \cite[Theorem 4.3.1]{SS}. Let $V$ be a finitely generated $\D$-module. Since $\Sh(\D^{\op}, \, J_a, \, \underline{k}) \simeq \D \module /\D \fdmod$, the kernel and cokernel of the natural map from $V$ to its sheafification $V^{\sharp}$ in $\Sh(\D^{\op}, \, J_a, \, \underline{k})$ are finite dimensional. But every finite dimensional $\D$-module has finite injective dimension. Therefore, we only need to show that $V^{\sharp}$ has finite injective dimension in $\D \module$.

Since $\Sh(\D^{\op}, \, J_a, \, \underline{k}) \simeq \C \fdmod$, $V^{\sharp}$ has a finite injective resolution
\[
V^{\sharp} \to I^0 \to \ldots \to I^s \to 0
\]
in  $\Sh(\D^{\op}, \, J_a, \, \underline{k})$. This is a finite complex in $\D \module$ for which all terms except $V^{\sharp}$ are injective. Moreover, all cohomology groups are finite dimensional $\D$-modules. Consequently, $V^{\sharp}$ as a $\D$-module has finite injective dimension as desired.
\end{proof}

Relations among these categories are illustrated in the following commutative diagram
\[
\xymatrix{
& \D \module \ar[d]^-{\mathrm{loc}} \ar[r]^-{\coind} \ar[ld]_-{\nu} & \C \module \ar[r]^-{\nu} \ar[d]^-{\mathrm{loc}} & \C \fdmod\\
\D \fdmod & \D \module / \D \fdmod \ar[d]^-{\simeq} \ar[l]^-{\bar{\nu}} & \C \module / \C \fdmod \ar[ur]_-{\bar{\nu}}^-{\simeq} \ar[d]^-{\simeq}\\
& \sh(\D^{\op}, \, J_a, \, \underline{k}) \ar[r]^-{\coind}_{\simeq} & \sh(\C^{\op}, \, J_a, \, \underline{k}),
}
\]
where $\mathrm{loc}$ is the localization functor, and $\bar{\nu}$ is the functor induced by the \textit{Nakayama functor} $\nu$ (see \cite{GLX} for details).

\begin{remark}
It is well known that every finitely generated $J_a$-torsion free injective $\C$-module is projective. This is not the case in $\D \Mod$. Moreover, the functor
\[
\bar{\nu}: \D \module / \D \fdmod \longrightarrow \D \fdmod
\]
is not an equivalence. For example, let $\D$ be obtained via removing the object $\boldsymbol{0}$ from $\C$. Then the constant $\D$-module $\underline{k}$ is injective and $J_a$-torsion free, but not projective. Furthermore, it is sent to the zero module in $\D \fdmod$ via $\bar{\nu}$.
\end{remark}

\begin{remark}
Based on the work of G\"{u}nt\"{u}rk\"{u}n-Snowden \cite{GS} and \cite{DLLX, GL1}, one can use sheaf theory to deduce similar results for the category $\OI$ of finite linearly ordered sets and order-preserving injections. For instance, if $k$ is a field of characteristic 0, and $\D$ is an infinite full subcategory of a skeleton $\C$ of $\OI$, then one has
\[
\D \module / \D \tmod \simeq \C \fdmod,
\]
where $\D \tmod$ is the category of finitely generated $J_a$-torsion $\D$-modules. Note that in this case finite dimensional modules are $J_a$-torsion, but not all finitely generated $J_a$-torsion modules are finite dimensional.
\end{remark}

\subsection{Representations of groups}

In this subsection we consider applications of sheaf theory in group representation theory. A bridge connecting these two areas is the following construction.

Let $G$ be a group, and $\mathcal{H}$ a family of subgroups of $G$. The \textit{orbit category} $\Orb(G, \mathcal{H})$ of $G$ with respect to $\mathcal{H}$ is defined as follows: objects are left cosets $G/H$ with $H \in \mathcal{H}$, and morphisms from $G/H$ to $G/K$ are $G$-equivariant maps. Every such morphism is an epimorphism, and can be represented by an element $g \in G$ such that $gHg^{-1} \subseteq K$. In particular, for every object $G/H$, its endomorphism monoid is a group isomorphic to $N_G(H)/H$, so $\Orb(G, \mathcal{H})$ is an EI category. For more details, the reader can refer to \cite[III.9]{MM}. Moreover, it is easy to see that $\Orb(G, \mathcal{H})$ is noetherian (resp., artinian) if and only if the poset $\mathcal{H}$ with respect to inclusion is noetherian (resp., artinian).

Representations of $G$ are closely related to sheaves of modules on $\Orb(G, \mathcal{H})$ equipped with certain special Grothendieck topologies. Therefore, one may apply sheaf theory to study representations of groups. We give a few examples appearing in the literature to illustrate this approach.

\begin{example}
Let $G$ be a topological group, and let $\mathcal{H}$ be a cofinal system of open subgroups. Impose the atomic Grothendieck topology $J_a$ on $\Orb(G, \mathcal{H})$, and let $\underline{k}$ be a constant structure sheaf given by a commutative ring $k$. Artin's theorem \cite[III.9 Theorem 1]{MM} asserts that
\[
\Sh(\Orb(G, \, \mathcal{H}), \, J_a, \underline{k}) \simeq kG \Mod^{\mathrm{dis}},
\]
the category of \textit{discrete representations} of $G$, which are $k$-modules equipped with the discrete topology such that $G$ acts on them continuously. In \cite{DLLX} the authors use this equivalence to study discrete representations of the infinite symmetric group $\varinjlim_n \Sigma_n$ and the infinite general linear group $\varinjlim_n \mathrm{GL}_n (\mathbb{F}_q)$ over a finite field. In particular, they obtain a classification of irreducible discrete representations when $k$ is a field of characteristic 0.
\end{example}

The following example is introduced by Balmer in \cite{Bal} and further studied by Xiong, Xu, and Zheng in \cite{XX, XZ}.

\begin{example}
Let $G$ be a finite group, $\mathcal{H}$ the set of all subgroups of $G$, $p$ a prime number and $k$ a field of characteristic $p$. Following \cite[Definition 5.8]{Bal}, we call a family of morphisms $\{G/H_i \to G/H \}$ a \textit{sipp-covering} if there exists a certain $H_i$ such that the index $[H: H_i]$ is prime to $p$. By \cite[Theorem 5.11]{Bal}, these sipp-coverings form a basis of the \textit{sipp} topology $J$ on $\Orb(G, \mathcal{H})$.

Since $\Orb(G, \mathcal{H})$ is a finite EI category, $J$ shall be rigid. Actually, an object $G/H$ is $J$-irreducible if and only if $H$ is a $p$-group. Indeed, if $H$ is a $p$-group (including the trivial group) and $\{G/H_i \to G/H \}$ is a sipp-covering, then every $H_i$ is a $p$-group (otherwise, there is no morphism from $G/H_i \to G/H$), and $[H: H_i]$ is prime to $p$ for a certain $H_i$ if and only if $H$ is conjugate to this $H_i$. But this means that $G/H$ is isomorphic to the corresponding $G/H_i$. Since $\Orb(G, \mathcal{H})$ is an EI category, it follows that all morphisms from this $G/H_i$ to $G/H$ are isomorphisms, so $G/H$ is a $J$-irreducible object.

On the other hand, suppose that $H$ is not a $p$-group and let $H_p$ be its Sylow $p$-subgroup. Then any morphism $G/H_p \to G/H$ forms a sipp-covering. It generates a sieve contained in $J(G/H)$ which is not the maximal sieve. Consequently, by letting $\mathscr{D}$ be the full subcategory of $\Orb(G, \mathcal{H})$ consisting of objects $G/H$ with $H$ a $p$-subgroup, one obtains the following equivalence
\[
\Sh(\Orb(G, \mathcal{H}), \, J, \, \underline{k}) \simeq k\D^{\op} \Mod,
\]
a result described in \cite[Subsection 5.3]{XX}.
\end{example}

A further generalization of the above example is:

\begin{example}
Let $G$ be an artinian group (for instance, the Pr\"{u}fer $p$-group), $\mathcal{H}$ a set of subgroups of $G$, and $k$ a commutative ring. The orbit category $\Orb(G, \mathcal{H})$ is an artinian EI category, so by Theorem \ref{rigid topologies} every Grothendieck topology $J$ on it is rigid. By Theorem \ref{classify sheaves}, one has
\[
\Sh(\Orb(G, \mathcal{H}), \, J, \, \underline{k}) \simeq k\D^{\op} \Mod,
\]
where $\D$ is the full subcategory of $\Orb(G, \mathcal{H})$ consisting of $J$-irreducible objects.
\end{example}

There already exist quite a few reformulations of Alperin's weight conjecture. In particular, Webb gives in \cite{Webb} a reformulation in terms of the standard stratification property of the category algebra of the orbit category. The previous example gives a sheaf theoretic reformulation. We end this paper with an illustration of this observation.

Let $G$ be a finite group, $p$ a prime number, and $k$ an algebraically closed field of characteristic $p$. A \textit{weight} is a pair $(H, L)$ with $H$ a $p$-subgroup (possibly trivial) of $G$ and $L$ is an irreducible projective $kN_G(H)/H$-module. Two weights $(H_1, L_1)$ and $(H_2, L_2)$ are equivalent if there exists $g \in G$ such that the conjugate $H_1^g$ coincides with $H_2$ and $L_1^g \cong L_2$. Alperin's weight conjecture asserts that the number of isomorphism classes of irreducible $kG$-modules is equal to the number of equivalence classes of weights.

Let $\mathcal{H}$ be the family of $p$-subgroups of $G$ and equip $\Orb(G, \mathcal{H})$ with the atomic Grothendieck topology $J_a$. Since only the object $G/1$ is $J_a$-irreducible, and $N_G(1)/1 \cong G \cong G^{\op}$, one obtains
\[
\Sh(\Orb(G, \mathcal{H}), \, J, \, \underline{k}) \simeq kG \Mod.
\]
Therefore, Alperin's weight conjecture can be reformulated as follows: the number of isomorphism classes of irreducible sheaves is equal to the number of equivalence classes of weights.

Since the automorphism group of $G/H$ is $N_G(H)/H$, a weight of $G$ can be viewed as a ``local" representation of $\Orb(G, \mathcal{H})^{\op}$; that is, a representation supported only on objects isomorphic to $G/H$. Instead, irreducible objects in $\Sh(\Orb(G, \mathcal{H}), \, J, \, \underline{k})$ can be viewed as ``global" representations of $\Orb(G, \mathcal{H})^{\op}$. Thus this reformulation reflects the ``local-global" principle of Alperin's weight conjecture. Furthermore, one may form a stronger version of this conjecture as follows: construct a bidirectional machinery such that from each weight one can obtain a unique irreducible sheaf up to isomorphism, and by a converse procedure one recovers this weight from the isomorphism class of the irreducible sheaf.

Let $\C = \Orb(G, \mathcal{H})^{\op}$ and $k\C$ the category algebra of $\C$. Note that ``local" representations and ``global" representations of $\C$ are related by the restriction and induction functors. Thus one may ask the following questions:
\begin{enumerate}
\item For each weight $(H, L)$, can we find a unique $J$-saturated irreducible $\C$-module $V(H, L)$ such that $V(H, L) \cong V(H', L')$ if and only if $(H, L)$ and $(H', L')$ are equivalent?

\item If the answer to the above question is yes, can every irreducible $J_a$-saturated $\C$-module  be obtained via the above construction?
\end{enumerate}

If one can answer these two questions affirmatively, then we can get a bijective correspondence between the number of equivalence classes of weights and the set of isomorphism classes of irreducible $J_a$-saturated $k\C$-modules. A proof of  Alperin's weight conjecture then follows from this bijection.

\end{document}